\documentclass[12pt]{amsart}

\usepackage{epsfig}
\usepackage{amsmath}
\usepackage{amssymb,amsthm}
\usepackage{graphicx}
\usepackage{booktabs}

\usepackage{pst-node}
\usepackage{tikz}
\usepackage{enumerate}
\usepackage[backref]{hyperref}
\usepackage{a4wide}
\usepackage[normalem]{ulem}

\usepackage[margin=2.9cm]{geometry}

\newtheorem{propo}{Proposition}[section]

\newtheorem{lemma}[propo]{Lemma}

\newtheorem{theo}[propo]{Theorem}

\newtheorem{prop}[propo]{Proposition}

\newtheorem{cond}[propo]{Condition}

\newcommand{\bl}{\begin{lemma}}
\newcommand{\el}{\end{lemma}}

\def\Cay{{\rm Cay}}

\usepackage{indentfirst,latexsym,bm}

\def\PSL{{\rm PSL}}
\def\PGL{{\rm PGL}}

\def\U{{\rm U}}
\def\Sym{{\rm Sym}}

\def\Aut{{\rm Aut}}

\def\K{{\rm K}}
\def\AGL{{\rm AGL}}

\begin{document}
\title{On  2-distance-transitive  circulant digraphs}

\thanks{Supported by NSFC (12271524, 12331013), NSF of Hunan (2026JJ50358) and the Scientific Research Project of the Education Department of Hunan Province (24A0142).}

\author[W. Jin]{Wei Jin}
\address{Wei Jin\\
School of Mathematics and Computational Science\\
Hunan Research Center of the Basic Discipline Fundamental Algorithmic Theory and Novel Computational Methods\\
Xiangtan University\\
Xiangtan, Hunan, 411105, P.R. China}
\email{jinwei@xtu.edu.cn}

\author[C. X. Li]{Cai Xia Li}
\address{Cai Xia Li\\
School of Mathematics and Computational Science\\
Xiangtan University\\
Xiangtan, Hunan, 411105, P.R. China}
\email{caixial@yeah.net}

\author[P. S. Li]{Ping Shan Li}
\address{Ping Shan Li\\
School of Mathematics and Computational Science\\
Hunan Research Center of the Basic Discipline Fundamental Algorithmic Theory and Novel Computational Methods\\
Xiangtan University\\
Xiangtan, Hunan, 411105, P.R. China}
\email{lips@xtu.edu.cn}

\author[X. L. Sun]{Xiao Lin Sun}
\address{Xiao Lin Sun\\
School of Mathematics and Computational Science\\
Xiangtan University\\
Xiangtan, Hunan, 411105, P.R. China}
\email{995164982@qq.com}

\author[J. Wu]{Jue Wu}
\address{Jue Wu\\
School of Mathematics and Computational Science\\
Xiangtan University\\
Xiangtan, Hunan, 411105, P.R. China}
\email{202631510198@smail.xtu.edu.cn}

\author[F. Yang]{Fan Yang}
\address{Fan Yang\\
School of Mathematics and Computational Science\\
Xiangtan University\\
Xiangtan, Hunan, 411105, P.R. China}
\email{202631510207@smail.xtu.edu.cn}


\maketitle

\begin{abstract}

Circulant digraphs form a prominent class of  Cayley digraphs defined on  finite cyclic groups. Building on the existing  classification of $2$-arc-transitive circulant graphs, this paper presents a complete classification of   $2$-distance-transitive circulant digraphs.
Our main theorem establishes   that  every connected $2$-distance-transitive circulant digraph  is isomorphic to one of the following:
the undirected cycle  $C_n$, the complete bipartite graph $\K_{\frac{n}{2},\frac{n}{2}}$, the complete multipartite graph $\K_{m[b]}$ with $m\geq 3,b\geq 2$,  the graph $\K_{\frac{n}{2},\frac{n}{2}}-\frac{n}{2}\K_2$ for  odd $\frac{n}{2}$,
prime-order Paley graphs,  the directed cycle $\overrightarrow{C}_n$, the oriented graph \(G(p^m,r)\) satisfying Condition~\ref{p-power-normal-2dt-cond},
the oriented graph  \(  C_r(b,1)\) with $r\geq 3,b\geq 2$ and $rb=n$,  the lexicographic product oriented graph \(  G(p^m,r)[\overline{\K}_d]\)
where \(G(p^m,r)\) obeys  Condition~\ref{p-power-normal-2dt-cond}.


\end{abstract}

\vspace{2mm}

\section{Introduction}

In this paper, we establish a complete classification of 2-distance-transitive  circulant digraphs, a prominent class of symmetric digraphs in algebraic graph theory. We begin by introducing the fundamental definitions, standard notation, and necessary research background to frame our investigation.

A finite \emph{digraph} (short for \emph{directed graph}) $\Gamma$ is a pair $(V(\Gamma),\operatorname{Arc}(\Gamma))$ consisting of a finite vertex set $V(\Gamma)$ and an arc set $\operatorname{Arc}(\Gamma)\subseteq V(\Gamma)\times V(\Gamma)$. An element $(u,v)\in \operatorname{Arc}(\Gamma)$ is a \emph{directed edge} from $u$ to $v$, denoted by $u\rightarrow v$, the ordered pair $(u,v)$ is then called an \emph{arc} of $\Gamma$. For a vertex $v\in V(\Gamma)$, the set of \emph{in-neighbours} of $v$ is $\Gamma^-(v)=\{u\in V\mid u\rightarrow v\}$, and the set of \emph{out-neighbours} of $v$ is $\Gamma^+(v)=\{u\in V\mid v\rightarrow u\}$. A digraph $\Gamma$ is called \emph{$k$-regular} if $|\Gamma^-(v)|=|\Gamma^+(v)|=k$ for every $v\in V(\Gamma)$,  it is \emph{regular} if it is $k$-regular for some positive integer $k$, and this integer $k$ is then the \emph{valency} of $\Gamma$.
An \emph{oriented graph} is a digraph with no 2-cycles (that is, no pairs of opposite arcs).
A \emph{graph} is a digraph in which every arc is bidirectional: for all vertices $u,v$, $(u,v)$ is an arc exactly when $(v,u)$ is an arc. If $\Gamma$ is a graph and $v\in V(\Gamma)$, then $\Gamma^+(v)=\Gamma^-(v)$, and we denote this common set by $\Gamma(v)$.

A digraph (graph) $\Gamma$ is called a \emph{circulant digraph} (\emph{circulant graph}) if its full automorphism group contains a cyclic subgroup that acts semiregularly and transitively on the vertex set. As one of the most fundamental and extensively studied families of algebraic graphs, circulant digraphs (graphs) possess inherent cyclic symmetry and highly regular structural properties. Owing to their elegant combinatorial features and wide-ranging applications in network design, coding theory and parallel computing, they have long been a central object in the study of vertex-transitive digraphs (graphs).

The study  of  circulant digraphs originated naturally from the analysis of primitive group actions. According to a classical theorem of Schur (see \cite[Theorem~25.2]{Wielandt-book}), every   circulant graph with a primitive  automorphism group acting  on its  vertex set is either a complete graph or a circulant graph of prime order. This landmark  result established  a  fundamental framework for  subsequent research and inspired extensive  studies on   circulant graphs under various refined group-theoretic constraints.

While primitivity imposes  a strong global symmetry condition, subsequent  research has focused on   the more flexible symmetry property of arc-transitivity. The systematic investigation  of arc-transitive circulant digraphs was initiated in the 1970s by Chao and Wells \cite{Chao1971,Chao1973}, who pioneered the exploration of symmetric circulant graphs and digraphs.
A series of influential  advances have greatly  enriched  this research field.  Alspach, Conder, Maru\v{s}i\v{c} and Xu  \cite[Theorem~1.1]{ACMX-1996}  completely classified  $2$-arc-transitive circulant graphs,  Li, Maru\v{s}i\v{c} and Morris \cite{LMM-circulant-2001} achieved a classification of  arc-transitive circulants of square-free order.  Xu, Baik and Sim \cite{XBS-2004} fully  characterized  arc-transitive circulant digraphs of odd prime-power order.   Kov\'{a}cs \cite{Kovacs-2004} and Li \cite{LCH-circulant-2005} independently established  reductive structural characterizations for  connected arc-transitive circulant digraphs. Furthermore,  Ara\'{u}jo, Bentz, Dobson, Konieczny and Morris \cite{ABDKM-2018} investigated the automorphism groups of circulant digraphs and completed  a classification for such digraphs with   small valency.  Chen, Jin and Li \cite{CJL-2019} settled the classification of  2-distance-transitive circulant graphs in 2018.
Building upon these milestone works, Li, Xia and Zhou \cite{LXZ-at-circ-2021} further  provided a refined and  explicit structural characterization of all arc-transitive circulant digraphs. Very  recently, the complete classification of finite locally-primitive circulant graphs was established  in \cite{JZ-2026}.

A digraph  $\Gamma$ is said to be  \emph{$2$-distance-transitive} if, for any two distinct pairs of vertices  $(u_1,v_1)$ and $(u_2,v_2)$ with directed distance $i$ ($i\in \{1,2\}$), there exists an  element of $\Aut(\Gamma)$ mapping $(u_1,v_1)$ to $(u_2,v_2)$. Every $2$-arc-transitive digraph is $2$-distance-transitive, while   the converse fails. For example, any   Paley tournament with at least 7 vertices  is $2$-distance-transitive but not $2$-arc-transitive. Although considerable  progress has been made on  primitive, 2-arc-transitive and 2-distance-transitive  circulant graphs, the classification of 2-distance-transitive circulant digraphs  remains an   open problem.

As a natural and meaningful extension of existing classification results, the characterization of  2-distance-transitive  circulant digraphs not only improves  the classification theory of symmetric circulant digraphs, but also provides new insights into the intrinsic interplay between local group actions and global graph symmetries. In this paper, we  establish a complete classification of 2-distance-transitive circulant digraphs.

Our main theorem is the following:

\begin{theo}\label{th-2dt-circ}
Let $T$ be a cyclic group of order $n\geq 3$ and let $S\subseteq T$ be such that $T=\langle S\rangle$. Suppose that  $\Gamma=\operatorname{Cay}(T,S)$ is a $2$-distance-transitive circulant digraph. Then  one of the following two cases holds.

\begin{enumerate}[{\rm (I)}]
\item $S=-S$, and $\Gamma$ is isomorphic to one of the following graphs:
\begin{itemize}
\item[(1)] the complete bipartite graph $\K_{\frac{n}{2},\frac{n}{2}}$;
\item[(2)] the complete multipartite graph $\K_{m[b]}$ with $m\geq 3,b\geq 2$;
\item[(3)] the $\K_{\frac{n}{2},\frac{n}{2}}-\frac{n}{2}\K_2$ with  $\frac{n}{2}$ is odd;
\item[(4)] the cycle  $C_n$;
\item[(5)] prime-order Paley graphs.
\end{itemize}

\item $S\cap -S=\varnothing$,    and $\Gamma$  is isomorphic to one of the following oriented graphs:
\begin{itemize}
\item[(1)]  the directed cycle $\overrightarrow{C}_n$;


\item[(2)] the oriented graph \(G(p^m,r)\) satisfying  Condition~\ref{p-power-normal-2dt-cond};

\item[(3)]  the oriented graph \(  C_r(b,1)\) with $r\geq 3,b\geq 2$ and $rb=n$;

\item[(4)] the lexicographic product oriented graph \(  G(p^m,r)[\overline{\K}_d]\)
where \(G(p^m,r)\) satisfies Condition~\ref{p-power-normal-2dt-cond}.
\end{itemize}

\end{enumerate}

In particular, the digraphs listed in cases (II) (1)-(2) are normal, while those  in cases (II) (3)-(4) are non-normal.

\end{theo}


\bigskip
\bigskip

\section{Preliminaries}

In this section we collect  some definitions concerning groups and digraphs,  and we also provide several results that will be used in the subsequent discussion.

Throughout this paper, all digraphs are finite, simple and connected. For group-theoretic terminology not defined here we refer the reader to \cite{Cameron-1,DM-1,Wielandt-book}.

\subsection{Groups }

Let $G$ be a permutation group on a set $\Omega$. We say that $G$ is \emph{semiregular} on $\Omega$ if $G_\alpha=1$ for every $\alpha\in\Omega$, and \emph{regular} if it is transitive and semiregular. The action of $G$ on $\Omega$ is \emph{faithful} if the only element fixing every point of $\Omega$ is the identity.

Let $G$ act transitively on $\Omega$. A \emph{partition} of $\Omega$ is a set $\mathcal B=\{B_1,B_2,\dots,B_n\}$ of non-empty subsets such that $B_i\cap B_j=\varnothing$ whenever $i\ne j$ and $\Omega=\bigcup_i B_i$. A partition $\mathcal B$ is \emph{$G$-invariant} if for every $g\in G$ and every $B_i\in\mathcal B$, we have $B_i^g\in\mathcal B$. The partitions into singletons and into one part are the \emph{trivial} partitions, while all other partitions are non-trivial. Each member of a $G$-invariant partition is called a \emph{block} of $G$. The whole set $\Omega$ and the singletons  are \emph{trivial blocks}, all others are \emph{non-trivial}. If $N$ is an intransitive normal subgroup of $G$, then each $N$-orbit is a block of $G$, and the set of $N$-orbits forms a $G$-invariant partition of $\Omega$. The action of $G$ on $\Omega$ is \emph{primitive} if it has no non-trivial $G$-invariant partitions, otherwise it is \emph{imprimitive}. There is a celebrated classification of finite primitive permutation groups into eight types, mainly due to O'Nan and Scott, see \cite{LPS-1}.

A transitive permutation group $G\le \operatorname{Sym}(\Omega)$ is \emph{quasiprimitive} if every non-trivial normal subgroup of $G$ is transitive on $\Omega$. Quasiprimitivity generalizes primitivity, since every normal subgroup of a primitive group is transitive, but there exist quasiprimitive groups that are not primitive. For more information about quasiprimitive permutation groups, we refer the reader to \cite{Praeger-1993-onanscott,Praeger-2,Praeger-2003-biq}.

 Let $T$ be a finite group and set  $\Omega=T$. For each $g\in
T$, define a  map of $T$ into itself as  follows:
$R(g): x\mapsto xg, \forall  x\in T.$
The map $R(g)$ is called the \emph{right multiplication} induced by
$g$. Furthermore,  $R(T)$ forms  a  subgroup of $\Sym(T)$ that
is  isomorphic to $T$ and  acts regularly on $T$.
 The \emph{holomorph} of \( T \), denoted by \( \operatorname{Hol}(T) \), is defined as  the semidirect product of \( R(T) \) and its automorphism group \( \operatorname{Aut}(T) \), namely
\(
\operatorname{Hol}(T) = R(T) \rtimes \operatorname{Aut}(T),
\)
where each automorphism \( \alpha \in  \operatorname{Aut}(T) \) acts naturally on \( T \)  by mapping \( x \in T \) to   \( \alpha(x) \).

We denote by $\mathbb Z_n$ the cyclic group of order $n$.

\subsection{Graphs}

Let $\Gamma$ be a digraph. Denote by $V(\Gamma)$, $\operatorname{Arc}(\Gamma)$ and $\operatorname{Aut}(\Gamma)$ its vertex set, arc set and automorphism group, respectively. The size of the vertex set is the \emph{order} of $\Gamma$. For a subgroup $G\le \operatorname{Aut}(\Gamma)$, let $G_v$ be the stabilizer of $v$ in $G$, and let $G_v^{\Gamma^+(v)}$ be the permutation group induced by $G_v$ on $\Gamma^+(v)$. For a positive integer $s$, an \emph{$s$-arc} of $\Gamma$ is a sequence $(v_0,v_1,\dots,v_s)$ such that $v_i$ and $v_{i+1}$ are adjacent and $v_{j-1}\ne v_{j+1}$ for $0\le i\le s-1$ and $1\le j\le s-1$. In particular, $1$-arcs are simply called \emph{arcs}.
The digraph $\Gamma$ is said to be \emph{$G$-vertex-transitive} or \emph{$(G,s)$-arc-transitive} if $G$ is transitive on the vertex set or on the set of $s$-arcs, respectively.
We say that $\Gamma$ is \emph{$(G,2)$-distance-transitive} if, for any two distinct vertex pairs $(u_1,v_1)$ and $(u_2,v_2)$ satisfying $d_{\Gamma}(u_1,v_1) = d_{\Gamma}(u_2,v_2) = i$ with $i\in \{1,2\}$, some element of $G$ maps $(u_1,v_1)$ to $(u_2,v_2)$.  We say that $\Gamma$ is \textit{$G$-distance-transitive} if for every integer $i \leq \operatorname{diam}(\Gamma)$, $G$ is transitive on the set of ordered vertex pairs at distance $i$.
If $G=\operatorname{Aut}(\Gamma)$, we drop the prefix ``$G$-''.


For a digraph $\Gamma$, its \emph{complement} $\overline{\Gamma}$ is the digraph with vertex set $V(\Gamma)$ and two vertices adjacent if and only if they are not adjacent in $\Gamma$.

For a positive integer $n$, $\K_n$ denotes the complete graph on $n$ vertices, and for $n\ge 3$, $C_n$ denotes a cycle of length $n$, while $\overrightarrow{C}_n$ denotes a directed cycle of length $n$. For $m,n\ge 2$, we denote by $\K_{m[n]}$ the complete multipartite graph with $m$ parts of size $n$,  that is, two vertices are adjacent if and only if they lie in distinct parts.

Let $\Gamma$ be a digraph and let $G\le\operatorname{Aut}(\Gamma)$. Suppose $\mathcal B=\{B_1,\dots,B_n\}$ is a $G$-invariant partition of $V(\Gamma)$. The \emph{quotient digraph} $\Gamma_{\mathcal B}$ of $\Gamma$ relative to $\mathcal B$ is the digraph with vertex set $\mathcal B$ such that $(B_i,B_j)$ is an arc  of $\Gamma_{\mathcal B}$ if and only if there exist $x\in B_i$, $y\in B_j$ with $(x,y)$ an arc  of $\Gamma$. We say that $\Gamma_{\mathcal B}$ is \emph{nontrivial} if $1<|\mathcal B|<|V(\Gamma)|$. Since $\mathcal B$ is $G$-invariant, $G$ induces a subgroup of automorphisms of $\Gamma_{\mathcal B}$.

Let $q=p^f$ be a prime power such that $q\equiv 1 \pmod{4}$. Let
$\mathbb{F}_q$ be the finite field of order $q$. Then the  \emph{Paley graph} of order $q$
 is the graph with vertex set  $\mathbb{F}_q$, and two distinct
vertices $u,v$ are adjacent if and only if $u-v$ is a nonzero square
in $\mathbb{F}_q$. The congruence condition on $q$ implies that $-1$ is a
square in $\mathbb{F}_q$, and hence each Paley graph  is an undirected graph.

Let $q=p^f$ be a prime power with $q\equiv 3\pmod 4$ and $q\ge 3$. The \emph{Paley tournament} $\mathrm{PTr}(q)$ is the directed graph with vertex set  $\mathbb F_q$, in which there is an arc from $x$ to $y$ (denoted $x\to y$) if and only if $y-x$ is a non-zero square in  $\mathbb F_q$. Explicitly,
\[
V(\mathrm{PTr}(q))=\mathbb F_q,\qquad
\operatorname{Arc}(\mathrm{PTr}(q))=\{(x,y)\in\mathbb F_q\times\mathbb F_q\mid y-x\in S\},
\]
where $S=\{s\in\mathbb F_q^\times\mid s\text{ is a square}\}$. Since $-1$ is a non-square when $q\equiv 3\pmod 4$, we have $S\cap(-S)=\varnothing$ and $S\cup(-S)=\mathbb F_q^\times$, hence for any two distinct vertices exactly one of $x\to y$ or $y\to x$ holds, so $\mathrm{PTr}(q)$ is a tournament. Equivalently, $\mathrm{PTr}(q)=\operatorname{Cay}(T,S)$ with $T=(\mathbb F_q,+)$, and its valency is $(q-1)/2$.

The automorphism group of $\mathrm{PTr}(q)$ contains all transformations of the form
\[
x\mapsto a x^\sigma + b,\qquad a\in S,\; \sigma\in\operatorname{Gal}(\mathbb F_q),\; b\in\mathbb F_q,
\]
that is,
\[
T\rtimes (S\rtimes\operatorname{Gal}(\mathbb F_q)) \le \operatorname{Aut}(\mathrm{PTr}(q)).
\]
Here $T$ is the translation group $\{x\mapsto x+b\}$ and $S$ is the multiplicative group of non-zero squares.

For a finite group $T$ and a subset $S\subseteq T$ with $1\notin S$, the \emph{Cayley digraph} $\operatorname{Cay}(T,S)$ of $T$ with respect to $S$ is the digraph with vertex set $T$ and arc set $\{(g,sg)\mid g\in T,\ s\in S\}$. In particular, $\operatorname{Cay}(T,S)$ is connected if and only if $T=\langle S\rangle$. The group $R(T)=\{\sigma_t\mid t\in T\}$ of right multiplications $\sigma_t:x\mapsto xt$ is a subgroup of $\operatorname{Aut}(\Gamma)$ and acts regularly on the vertex set. Indeed, a digraph is a Cayley digraph if and only if it admits a regular group of automorphisms. For a Cayley digraph $\Gamma=\operatorname{Cay}(T,S)$, let
\[
\operatorname{Aut}(T,S)=\{\alpha\in\operatorname{Aut}(T)\mid S^\alpha=S\}.
\]
It was shown in \cite{Godsil1983} that the normalizer of $R(T)$ in $\operatorname{Aut}(\Gamma)$ is $R(T)\rtimes\operatorname{Aut}(T,S)$. If $R(T)\trianglelefteq \operatorname{Aut}(\operatorname{Cay}(T,S))$, then $\operatorname{Cay}(T,S)$ is called a \emph{normal Cayley digraph} of $T$ (see \cite{Xu98}). Circulant digraphs are precisely the Cayley digraphs on  cyclic groups, a circulant digraph is normal if its automorphism group contains a normal regular subgroup.


\subsection{Some lemmas}

The classification of    primitive permutation groups  that contain a
cyclic regular subgroup was independently obtained by Jones \cite{Jones-2002} and Li \cite[Corollary 1.2]{LCH-abelianregular-2003}.
Moreover, by \cite[Theorem 1.2]{LP-circulant-2012}, every   quasiprimitive group with a regular cyclic  subgroup is  primitive.
Hence the family of  quasiprimitive groups with a regular cyclic  subgroup is also  completely  determined in the following lemma.

\begin{lemma}{\rm (\cite{Jones-2002},\cite[Corollary 1.2]{LCH-abelianregular-2003})}\label{primitive-cyclic-1}
A primitive permutation group $G$ of degree $n$  contains a
cyclic regular subgroup if and only if one of the  following holds.
 \begin{enumerate}[{\rm (i)}]
\item $\mathbb{Z}_p\leqslant G\leqslant \AGL(1,p)$, where $n=p$ is a prime;
\item $G=A_n$ with $n\geqslant 5$ odd, or $S_n$, where $n\geqslant 4$;
\item $\PGL(d,q)\leqslant G\leqslant P\Gamma L(d,q)$ and $n=(q^d-1)/(q-1)$;
\item $(G,n)=(\PSL(2,11),11)$, $(M_{11},11)$, $(M_{23},23)$.
\end{enumerate}

Moreover, in cases {\rm (ii)}--{\rm (iv)} $G$ is $2$-transitive.
\end{lemma}

We  recall a classic characterization theorem for normal Cayley digraphs.

\begin{lemma}{\rm (\cite[Propositions 1.3 and 1.5]{Xu98})}\label{cayley-normal}
The digraph $\Gamma=\operatorname{Cay}(T,S)$ is a normal Cayley digraph if and only if $\operatorname{Aut}(\Gamma)_1=\operatorname{Aut}(T,S)$, equivalently $\operatorname{Aut}(\Gamma)=T\rtimes\operatorname{Aut}(T,S)$, where $\operatorname{Aut}(\Gamma)_1$ is the stabilizer of the identity vertex.
\end{lemma}

\medskip
The \emph{lexicographic product} $\Gamma_1[\Gamma_2]$ of two digraphs $\Gamma_1$ and $\Gamma_2$ is the digraph with vertex set $V(\Gamma_1)\times V(\Gamma_2)$ such that $(u_1,u_2)$ points to   $(v_1,v_2)$ if and only if either $(u_1,v_1)$ is an arc  of $\Gamma_1$, or $u_1=v_1$ and $(u_2,v_2)$ is an arc  of $\Gamma_2$.

For a positive integer $b$ and a digraph $\Gamma$, denote by $b\Gamma$ the digraph consisting of $b$ vertex-disjoint copies of $\Gamma$. The digraph $\Gamma[\overline{\K_b}]-b\Gamma$ has the same vertex set as $\Gamma[\overline{\K_b}]$ and its edge set is obtained by removing all edges of $b\Gamma$ from those of $\Gamma[\overline{\K_b}]$.

For two digraphs $\Gamma_1$ and $\Gamma_2$, the \emph{direct product} (or \emph{tensor product}) $\Gamma_1\times\Gamma_2$ is the digraph with vertex set $V(\Gamma_1)\times V(\Gamma_2)$ such that $(u_1,u_2)$ points to   $(v_1,v_2)$ if and only if $(u_1,v_1)$ is an arc of $\Gamma_1$ and $(u_2,v_2)$ is an arc of $\Gamma_2$. By definition, $\Gamma[\overline{\K_b}]-b\Gamma\cong \Gamma\times\K_b$.

The following result gives a characterization of arc-transitive circulant digraphs.

\begin{theo}{\rm (\cite[Theorem 1.1]{LXZ-at-circ-2021})}\label{arccirculant-explicit-1}
For every connected arc-transitive circulant $\Gamma$ of order $n$, there exist a connected arc-transitive normal circulant $\Gamma_0$ of order $n_0$ and positive integers $n_1,\dots,n_r,b$, where $r\ge 0$, such that the following hold:
\begin{enumerate}[{\rm (1)}]
    \item $\Gamma_0 \not\cong C_4$;
    \item $n_i \ge 4$ for $i=1,\dots,r$;
    \item $n = n_0 n_1 \cdots n_r b$, and $n_0,n_1,\dots,n_r$ are pairwise coprime;
    \item $\Gamma \cong (\Gamma_0 \times \K_{n_1} \times \cdots \times \K_{n_r})[\overline{\K}_b]$;
    \item $\operatorname{Aut}(\Gamma) \cong S_b \wr (\operatorname{Aut}(\Gamma_0) \times S_{n_1} \times \cdots \times S_{n_r})$.
\end{enumerate}
Moreover, $\Gamma$ is uniquely determined by the triple $(\Gamma_0, \{n_1,\dots,n_r\}, b)$ satisfying the above conditions.
\end{theo}

Let \(\Gamma\) be a connected \(2\)-distance-transitive circulant digraph. Let
\(
\Gamma\cong (\Gamma_0\times \K_{n_1}\times\cdots\times \K_{n_r})[\overline{\K}_b]
\)
be the decomposition given by Theorem \ref{arccirculant-explicit-1}. Then the following lemma shows that \(\Gamma_0\) is \(2\)-distance-transitive.

\begin{lemma}\label{prop:Gamma0-2DT}
Let \(\Gamma\) be a connected \(2\)-distance-transitive circulant digraph. Suppose that
\(
\Gamma\cong (\Gamma_0\times \K_{n_1}\times\cdots\times \K_{n_r})[\overline{\K}_b]
\)
is  the decomposition given by Theorem \ref{arccirculant-explicit-1}. Then \(\Gamma_0\) is \(2\)-distance-transitive.
\end{lemma}

\begin{proof}
Set
\(
H:=\Gamma_0\times \K_{n_1}\times\cdots\times \K_{n_r}.
\)
Then \(\Gamma\cong H[\overline{\K}_b]\). By Theorem \ref{arccirculant-explicit-1},
\(
\operatorname{Aut}(H)\cong \operatorname{Aut}(\Gamma_0)\times S_{n_1}\times\cdots\times S_{n_r},
\)
and
\(
\operatorname{Aut}(\Gamma)\cong S_b\wr \operatorname{Aut}(H).
\)

First we show that \(H\) is \(2\)-distance-transitive. Let
\(h_1,h_2,k_1,k_2\in V(H)\) be such that
\(
d_H(h_1,h_2)=d_H(k_1,k_2)=2.
\)
Choose arbitrary \(i,j,i',j'\in\{1,\dots,b\}\). Since \(\overline{\K}_b\) has no arcs, for \(h\neq h'\),
\[
d_\Gamma((h,i),(h',j))=d_H(h,h').
\]
Hence the pairs
\[
((h_1,i),(h_2,j))
\quad\text{and}\quad
((k_1,i'),(k_2,j'))
\]
are ordered pairs of vertices of \(\Gamma\) at distance \(2\). Since \(\Gamma\) is \(2\)-distance-transitive, there exists
\(
\Phi\in\operatorname{Aut}(\Gamma)
\)
such that
\[
\Phi(h_1,i)=(k_1,i'),\qquad
\Phi(h_2,j)=(k_2,j').
\]
Because \(\operatorname{Aut}(\Gamma)\cong S_b\wr \operatorname{Aut}(H)\), the automorphism \(\Phi\) induces an automorphism \(\phi\in\operatorname{Aut}(H)\) satisfying
\[
\phi(h_1)=k_1,\qquad
\phi(h_2)=k_2.
\]
Thus \(H\) is \(2\)-distance-transitive.

Now let \(u,u',v,v'\in V(\Gamma_0)\) satisfy
\[
d_{\Gamma_0}(u,u')=d_{\Gamma_0}(v,v')=2.
\]
Fix vertices \(0_i\in V(\K_{n_i})\) for \(i=1,\dots,r\), and put
\[
h_1=(u,0_1,\dots,0_r),\qquad
h_2=(u',0_1,\dots,0_r),
\]
\[
k_1=(v,0_1,\dots,0_r),\qquad
k_2=(v',0_1,\dots,0_r).
\]
The projection \(H\to \Gamma_0\) is a graph homomorphism, so
\[
d_H(h_1,h_2)\ge d_{\Gamma_0}(u,u')=2.
\]
Conversely, a directed path of length \(2\) from \(u\) to \(u'\) in \(\Gamma_0\) can be lifted to a directed path of length \(2\) from \(h_1\) to \(h_2\) in \(H\) by choosing suitable vertices in the complete factors. Hence
\(
d_H(h_1,h_2)=2.
\)
Similarly,
\(
d_H(k_1,k_2)=2.
\)

Since \(H\) is \(2\)-distance-transitive, there exists \(\phi\in\operatorname{Aut}(H)\) such that
\[
\phi(h_1)=k_1,\qquad
\phi(h_2)=k_2.
\]
Using
\[
\operatorname{Aut}(H)\cong \operatorname{Aut}(\Gamma_0)\times S_{n_1}\times\cdots\times S_{n_r},
\]
we may write
\[
\phi=(\alpha,\sigma_1,\dots,\sigma_r)
\]
with \(\alpha\in\operatorname{Aut}(\Gamma_0)\) and \(\sigma_i\in S_{n_i}\). From \(\phi(h_1)=k_1\) we obtain
\[
\alpha(u)=v,
\]
and from \(\phi(h_2)=k_2\) we obtain
\[
\alpha(u')=v'.
\]
Thus
\[
\alpha(u)=v,\qquad
\alpha(u')=v'.
\]
Therefore \(\operatorname{Aut}(\Gamma_0)\) acts transitively on ordered pairs of vertices at distance \(2\).

Finally, by Theorem \ref{arccirculant-explicit-1}, \(\Gamma_0\) is arc-transitive, so \(\operatorname{Aut}(\Gamma_0)\) acts transitively on ordered pairs at distance \(1\). It is trivially transitive on ordered pairs at distance \(0\). Hence \(\Gamma_0\) is \(2\)-distance-transitive.
\end{proof}


The following graph is defined in Definition 2.6 of \cite{Praeger-1989}.
Let $v\geq 2,r\geq 3$ and $s\geq 1$ be integers. Define $C_r(v,s)$ to be the digraph with vertex set
$\mathbb{Z}_r\times \mathbb{Z}_v^s$, and $(i,x)\rightarrow (j,y)$ for $x=(x_1,x_2,\ldots,x_s)$
and $y=(y_1,y_2,\ldots,y_s)\in \mathbb{Z}_v^s$, if and only if $j=i+1$ and $y=(y_1,x_1,x_2,\ldots,x_{s-1})$.
By \cite[Theorem 2.8]{Praeger-1989},    the digraphs $C_r(v,s)$, for $r\geq 3,v\geq 2,s\geq 1$, have valency $v$, diameter $r$ and are connected.
For all $s\geq 1$, the digraph $C_r(v,s)$ has automorphism group $G=S_v\wr \mathbb{Z}_r$, and for $1\leq s<r$, the digraph $C_r(v,r-s)$ is $(G,s)$-arc transitive but not $(G,s+1)$-arc transitive.
In particular, $C_r(v,1)\cong \overrightarrow{C_r}[\overline{\K_v}]$.

The following lemma shows that $C_r(v,1)$ is isomorphic to a circulant digraph.
\begin{lemma}\label{crv-circdig-1}
Let \(r\ge 3\) and \(v\ge 2\). Then \(C_r(v,1)\) is a Cayley digraph on the cyclic group \(\mathbb Z_{rv}\). More precisely,
\(
C_r(v,1)\cong \operatorname{Cay}\bigl(\mathbb Z_{rv},S\bigr),
\)
where
\(
S=\{1+kr \pmod{rv}: k=0,1,\dots,v-1\}.
\)
\end{lemma}

\begin{proof}
Write \(\Gamma=C_r(v,1)\). Its vertex set is
\(
V(\Gamma)=\mathbb Z_r\times \mathbb Z_v,
\)
and its arcs are
\(
(i,x)\longrightarrow (i+1,y)
\)
for all \(i\in\mathbb Z_r\) and all \(x,y\in\mathbb Z_v\).

Choose a \(v\)-cycle \(\sigma\in S_v\), for example \(\sigma(x)=x+1\pmod v\). Define a permutation \(g\) of \(V(\Gamma)\) by
\[
g(i,x)=
\begin{cases}
(i+1,\sigma(x)), & i=0,\\
(i+1,x), & i\neq 0.
\end{cases}
\]
Equivalently, if we set \(\sigma_0=\sigma\) and \(\sigma_i=\operatorname{id}\) for \(i\neq 0\), then
\[
g(i,x)=(i+1,\sigma_i(x)).
\]

First we show that \(g\in\operatorname{Aut}(\Gamma)\). Let \((i,x)\to(i+1,y)\) be an arc of \(\Gamma\). Then
\[
g(i,x)=(i+1,\sigma_i(x)),\qquad
g(i+1,y)=(i+2,\sigma_{i+1}(y)).
\]
Since every vertex in layer \(i+1\) is adjacent to every vertex in layer \(i+2\), the pair
\[
(i+1,\sigma_i(x))\longrightarrow (i+2,\sigma_{i+1}(y))
\]
is again an arc of \(\Gamma\). Hence \(g\) preserves arcs.

Next we determine the order of \(g\). For any \((i,x)\in V(\Gamma)\),
\[
g^r(i,x)=(i,\sigma(x)).
\]
Indeed, during \(r\) consecutive applications of \(g\), the first coordinate runs through all residues modulo \(r\), and the second coordinate is permuted by \(\sigma\) exactly when the current first coordinate is \(0\). Thus exactly one application of \(\sigma\) occurs. Since \(\sigma\) is a \(v\)-cycle, \(g^r\) has order \(v\).

Now suppose \(g^k=\operatorname{id}\). Then the first coordinate forces \(k\equiv 0\pmod r\), say \(k=rm\). Then
\[
g^k=g^{rm}=(g^r)^m
\]
acts on the second coordinate by \(\sigma^m\). Hence \(\sigma^m=\operatorname{id}\), so \(v\mid m\). Therefore \(rv\mid k\), and \(g\) has order \(rv\).

Moreover, if \(g^k(i,x)=(i,x)\), then the same argument gives \(k\equiv 0\pmod r\), say \(k=rm\), and then \(\sigma^m(x)=x\). Since \(\sigma\) is a \(v\)-cycle, this forces \(v\mid m\), hence \(rv\mid k\). Thus no nontrivial power of \(g\) fixes a vertex. Therefore \(\langle g\rangle\) is a regular cyclic group of order \(rv\) acting on \(V(\Gamma)\).

Since \(\langle g\rangle\) is a cyclic group of order \(rv\) acting regularly and as a group of automorphisms of \(\Gamma\), the digraph \(\Gamma\) is a Cayley digraph on \(\mathbb Z_{rv}\).

To identify the connection set, take the base vertex \((0,0)\). Its out-neighbors in \(\Gamma\) are exactly
\[
(1,y),\qquad y\in\mathbb Z_v.
\]
For \(k=0,1,\dots,v-1\),
\[
g^{1+kr}(0,0)=(1,\sigma^{k+1}(0)).
\]
Since \(\sigma\) is a \(v\)-cycle, the elements \(\sigma^{k+1}(0)\) run through all of \(\mathbb Z_v\). Hence the out-neighbors of \((0,0)\) correspond precisely to the exponents
\(
1+kr \pmod{rv}, k=0,1,\dots,v-1.
\)
Therefore, under the identification
\[
\mathbb Z_{rv}\longrightarrow V(\Gamma),\qquad
m\longmapsto g^m(0,0),
\]
we obtain
\[
C_r(v,1)\cong \operatorname{Cay}\bigl(\mathbb Z_{rv},S\bigr),
\]
where
\(
S=\{1+kr \pmod{rv}: k=0,1,\dots,v-1\}.
\)
This proves the claim.
\end{proof}

By Lemma \ref{crv-circdig-1}, \(C_r(v,1)\) is a Cayley digraph on the cyclic group
\(\mathbb Z_{rv}\), and the following results shows that it is not normal.

\begin{lemma}\label{crv-not-circdig-1}
Let \(r\ge 3\) and \(v\ge 2\). Then \(C_r(v,1)\) is a Cayley digraph on the cyclic group
\(\mathbb Z_{rv}\), but it is not a normal Cayley digraph on \(\mathbb Z_{rv}\).
More precisely, the cyclic regular subgroup \(H\cong \mathbb Z_{rv}\) constructed
below is not normal in \(\operatorname{Aut}(C_r(v,1))\).
\end{lemma}

\begin{proof}
Let \(\Gamma=C_r(v,1)\). By \cite[Theorem 2.8]{Praeger-1989},
\(
\operatorname{Aut}(\Gamma)=W=S_v\wr \mathbb Z_r.
\)
Write
\[
W=B\rtimes T,\qquad B=(S_v)^r,\qquad T=\langle t\rangle\cong \mathbb Z_r.
\]
An element of \(W\) is written as
\[
(\sigma_0,\sigma_1,\ldots,\sigma_{r-1};t),
\]
and it acts on \(V(\Gamma)=\mathbb Z_r\times \mathbb Z_v\) by
\[
(i,x)\longmapsto (i+t,\sigma_i(x)).
\]

Choose a \(v\)-cycle \(\sigma\in S_v\), and define
\[
g=(\sigma,1,\ldots,1;1)\in W.
\]
Then \(g\) has order \(rv\), and \(H=\langle g\rangle\) is a regular cyclic subgroup
of \(W\) of order \(rv\). Hence \(\Gamma\) is a Cayley digraph on
\(\mathbb Z_{rv}\) via the identification
\[
\mathbb Z_{rv}\longrightarrow V(\Gamma),\qquad
m\longmapsto g^m(0,0).
\]
We now show that \(H\) is not normal in \(W=\operatorname{Aut}(\Gamma)\).

First compute \(g^r\). Since \(g\) shifts the \(\mathbb Z_r\)-coordinate by \(1\)
and applies \(\sigma\) exactly once to each layer during \(r\) steps, we have
\[
g^r=(\sigma,\sigma,\ldots,\sigma;0)\in B.
\]
Thus
\[
H\cap B=\langle g^r\rangle
=\langle(\sigma,\sigma,\ldots,\sigma;0)\rangle.
\]

\medskip
\noindent
\textbf{Case 1: \(v\ge 3\).}

Since \(\sigma\) is a \(v\)-cycle and \(v\ge 3\), there exists \(b_0\in S_v\)
such that
\[
b_0\sigma b_0^{-1}=\sigma^{-1}.
\]
Let
\[
b=(b_0,1,\ldots,1;0)\in B.
\]
Then
\[
b(\sigma,\sigma,\ldots,\sigma;0)b^{-1}
=(\sigma^{-1},\sigma,\ldots,\sigma;0).
\]
This element is not in \(H\cap B\). Indeed, if it were in
\(\langle(\sigma,\sigma,\ldots,\sigma;0)\rangle\), then for some \(m\) we would have
\[
\sigma^{-1}=\sigma^m
\quad\text{and}\quad
\sigma=\sigma^m.
\]
The second equality gives \(m\equiv 1\pmod v\), and then the first gives
\(\sigma^{-1}=\sigma\), so \(\sigma^2=1\), contradicting that \(\sigma\) is a
\(v\)-cycle with \(v\ge 3\).

Thus \(H\cap B\) is not normal in \(W\). Since \(B\) is normal in \(W\), if
\(H\) were normal in \(W\), then \(H\cap B\) would also be normal in \(W\).
Therefore \(H\) is not normal in \(W\).

\medskip
\noindent
\textbf{Case 2: \(v=2\).}

Let \(\sigma\) be the nontrivial element of \(S_2\). Then
\(
g=(\sigma,1,\ldots,1;1).
\)
Let
\(
b=(\sigma,1,\ldots,1;0)\in B.
\)
Since \(b^2=1\) and \(g=t b\), we have
\[
b g b^{-1}=b(t b)b^{-1}=b t.
\]
A direct computation gives
\[
b g b^{-1}=(1,\ldots,1,\sigma;1),
\]
where the entry \(\sigma\) occurs in the last coordinate \(r-1\).

Suppose, for contradiction, that \(b g b^{-1}\in H=\langle g\rangle\). Then
\[
b g b^{-1}=g^k
\]
for some \(k\). Comparing the \(T\)-components gives \(k\equiv 1\pmod r\).
Write \(k=1+mr\). Since
\[
g^r=(\sigma,\ldots,\sigma;0)
\]
is central in \(W\) when \(v=2\), we obtain
\[
g^{1+mr}
=g(g^r)^m
=(\sigma^{m+1},\sigma^m,\ldots,\sigma^m;1).
\]
Thus we would have
\[
(1,\ldots,1,\sigma;1)
=
(\sigma^{m+1},\sigma^m,\ldots,\sigma^m;1).
\]
Now compare the entry in coordinate \(1\). Since \(r\ge 3\), the coordinate
\(1\) is distinct from \(0\) and from \(r-1\). Hence the left-hand side has
entry \(1\) in coordinate \(1\), while the right-hand side has entry
\(\sigma^m\) there. Therefore \(\sigma^m=1\), so \(m\) is even.

Next compare the entry in coordinate \(0\). The left-hand side has entry \(1\)
in coordinate \(0\), while the right-hand side has entry \(\sigma^{m+1}\).
Therefore \(\sigma^{m+1}=1\), so \(m+1\) is even, i.e. \(m\) is odd. This is a
contradiction.

Hence \(b g b^{-1}\notin H\). Therefore \(H\) is not normal in
\(W=\operatorname{Aut}(\Gamma)\).

\medskip
In both cases, the cyclic regular subgroup \(H\cong \mathbb Z_{rv}\) is not
normal in \(\operatorname{Aut}(C_r(v,1))\). Consequently, \(C_r(v,1)\) is not
a normal Cayley digraph on \(\mathbb Z_{rv}\).
\end{proof}


Let $p$ be an odd prime and   $m\geq 1$ an integer. Let $\mathbb{Z}_{p^m}$ be a cyclic group under addition.
Note that $\mathrm{Aut}(\mathbb{Z}_{p^m})\cong \mathbb{Z}_{p^{m-1}}\times \mathbb{Z}_{p-1}$. Let  $r$ be a positive divisor of $p-1$.  We use $H_r$ to denote the unique subgroup of $\mathrm{Aut}(\mathbb{Z}_{p^m})$ of order $r$, which is isomorphic to $\mathbb{Z}_r$.
We define a digraph $G(p^m,r)$ by
\begin{align*}
V(G(p^m,r))&=\mathbb{Z}_{p^m},\\
E(G(p^m,r))&=\{(x,y)\mid y-x\in H_r\}.
\end{align*}

The digraph $G(p^m,r)$ is isomorphic to the following Cayley digraph
\[
G(p^m,r) = \operatorname{Cay}(\mathbb{Z}_{p^m}, H_r).
\]

\begin{theo}{\rm(\cite[Theorem 4.2]{XBS-2004})}\label{at-primepower-normalcirc}
\begin{enumerate}[{\rm (1)}]
\item $G(p^m,r)$ is an arc-transitive digraph of order $p^m$ with out-valency and in-valency $r$. $G(p^m,r)$ is undirected if and only if $r$ is even.
\item Every normal arc-transitive circulant digraph of order $p^m$ is isomorphic to $G(p^m,r)$ for some divisor $r$ of $p-1$, and with $(p^m,r)\neq (p,p-1)$. (Note that $G(p,p-1)$ is complete, so it is not normal.)
\item $A=\mathrm{Aut}(G(p^m,r))\cong \mathbb{Z}_{p^m}\rtimes H_r$ acts regularly on the set of arcs for either $m>1$, or $m=1$ and $r<p-1$. So normal arc-transitive circulant (di)graphs are $r$-regular.
\end{enumerate}
\end{theo}

\bigskip
\bigskip

\section{ Normal $2$-distance-transitive  circulant digraphs of even order}

This section is devoted to the classification of all 2-distance-transitive  normal circulant digraphs of even order.

Let $\Gamma$ be a finite arc-transitive oriented graph of diameter 2. Suppose that there exists an arc $(u,v)$ such that the out-neighbourhood of $v$ equals the set of vertices at directed distance 2 from $u$. As shown in the first lemma,  $\Gamma$ is isomorphic to the directed 3-cycle.

\begin{lemma}\label{lem-1=2-diam=2}
Let \(\Gamma\) be a finite arc-transitive oriented graph of diameter \(2\). Suppose that there is an arc \((u,v)\) such that
\(
\Gamma^+(v)=\Gamma_2^+(u).
\)
Then \(
\Gamma\cong \overrightarrow{C_3}.
\)
\end{lemma}

\begin{proof}
Let
\(
m=|\Gamma^+(u)|
\)
be the out-valency of \(u\). Since \(\Gamma\) is arc-transitive, every vertex has out-valency and in-valency \(m\). By hypothesis,
\(
|\Gamma_2^+(u)|=|\Gamma^+(v)|=m.
\)
Because \(\Gamma\) has directed diameter \(2\), every vertex lies in
\[
\{u\}\cup \Gamma^+(u)\cup \Gamma_2^+(u).
\]
These three sets are pairwise disjoint: \(u\notin \Gamma^+(u)\) since there are no loops, \(u\notin \Gamma_2^+(u)\) since otherwise there would be a directed \(2\)-cycle, and \(\Gamma^+(u)\cap \Gamma_2^+(u)=\varnothing\) by definition. Hence
\[
|V(\Gamma)|=1+m+m=2m+1.
\]

By arc-transitivity, the condition \(\Gamma^+(v)=\Gamma_2^+(u)\) for the arc \((u,v)\) implies that for every arc \((x,y)\),
\[
\Gamma^+(y)=\Gamma_2^+(x). \tag{1}
\]
In particular, for every \(a\in \Gamma^+(u)\), applying (1) to the arc \((u,a)\) gives
\(
\Gamma^+(a)=\Gamma_2^+(u).
\)
Thus every vertex of \(\Gamma^+(u)\) points to every vertex of \(\Gamma_2^+(u)\). Consequently, for each \(b\in \Gamma_2^+(u)\), the in-neighbourhood of \(b\) contains \(\Gamma^+(u)\). Since both sets have size \(m\), we obtain
\[
\Gamma^-(b)=\Gamma^+(u) \qquad \text{for all } b\in \Gamma_2^+(u). \tag{2}
\]

Now fix \(b\in \Gamma_2^+(u)\). We determine its out-neighbours. If \(b\to c\) for some \(c\in \Gamma_2^+(u)\), then \(b\in \Gamma^-(c)\), contradicting (2), since \(b\notin \Gamma^+(u)\). If \(b\to a\) for some \(a\in \Gamma^+(u)\), then \(a\to b\) because \(\Gamma^+(a)=\Gamma_2^+(u)\), so \(a\leftrightarrow b\) is a directed \(2\)-cycle, which is forbidden. Therefore the only possible out-neighbour of \(b\) is \(u\). Hence
\(
|\Gamma^+(b)|\le 1.
\)
But \(|\Gamma^+(b)|=m\), so \(m\le 1\). Since there is an arc, \(m\ge 1\). Thus
\(
m=1.
\)
Therefore
\(
|V(\Gamma)|=2m+1=3.
\)
Write
\[
\Gamma^+(u)=\{v\},\qquad \Gamma_2^+(u)=\{w\}.
\]
Then \(u\to v\) and \(v\to w\). From the argument above, the only possible out-neighbour of \(w\) is \(u\), so \(w\to u\). Hence
\(
u\to v\to w\to u
\)
is a directed cycle of length \(3\), and \(\Gamma\) has exactly these three vertices and three arcs. Therefore
\(
\Gamma\cong \overrightarrow{C_3}.
\)
This completes the proof.
\end{proof}

\begin{lemma}\label{prop-1=2}
Let $\Gamma$ be a finite arc-transitive oriented graph of diameter at
least $2$. Suppose that there is an arc
$(u,v)$ such that
\( \Gamma^+(v)=\Gamma_2^+(u).\)
Then either $\Gamma$ is a directed cycle or
\(\Gamma\cong C_r(m,1)
\)
for some integers $m\ge 2$ and $r\ge 3$.
\end{lemma}

\begin{proof}
If $\Gamma$ has   diameter  $2$, then by Lemma \ref{lem-1=2-diam=2}, \(\Gamma\cong \overrightarrow{C_3}.\)
Henceforth, we suppose  that $\Gamma$ has  diameter at
least $3$.

If the out-valency of $\Gamma$ is   $1$, then $\Gamma$ is a directed cycle. In the remainder, assume
that the out-valency   is at least $2$, and set
\(m=|\Gamma^+(u)|. \)
Since $\Gamma$ is arc-transitive, every vertex has out-valency and in-valency $m$. Moreover,
\[
|\Gamma_2^+(u)|=|\Gamma^+(v)|=m.
\]
Because $G=\Aut(\Gamma)$ is transitive on arcs, the condition
$\Gamma^+(v)=\Gamma_2^+(u)$ implies that for every arc $(x,y)$,
\[
\Gamma^+(y)=\Gamma_2^+(x). \tag{1}
\]
In particular, for every $z\in\Gamma^+(u)$, we have
\( \Gamma^+(z)=\Gamma_2^+(u).\)
Thus
\( \Gamma^+(u)\to\Gamma_2^+(u) \)
is a complete bipartite oriented graph, and both layers have size $m$.

Let $d$ be the diameter of $\Gamma$, and let
\( (u=u_0,u_1,\dots,u_d) \)
be a directed $d$-geodesic. Since $(u,u_1)$ is an arc, by (1),
\(
\Gamma^+(u_1)=\Gamma_2^+(u).
\)

Let $t$ be the largest integer with $2\le t\le d$ such that
\[
\Gamma_j^+(u)=\Gamma^+(u_{j-1})
\]
for every $j=2,\dots,t$. Such a $t$ exists because the condition holds for
$j=2$. We shall prove that $t=d-1$.

\medskip

\noindent\textbf{Claim 1.}
For every $2\le j\le t$, the following hold:
\begin{itemize}
\item[(i)] $|\Gamma_j^+(u)|=m$;
\item[(ii)] $\Gamma_{j-1}^+(u)\to\Gamma_j^+(u)$ is a complete bipartite
oriented digraph;
\item[(iii)] for every $x\in\Gamma_j^+(u)$,
\( \Gamma^-(x)=\Gamma_{j-1}^+(u). \)
\end{itemize}

\emph{Proof of Claim 1.}
We prove the claim by induction on $j$. For $j=2$, the statements follow from
the observation above. Suppose the claim holds for $j-1$, where $3\le j\le t$.
Then \( \Gamma_j^+(u)=\Gamma^+(u_{j-1}). \)
Since $u_{j-1}\in\Gamma_{j-1}^+(u)$ and
$\Gamma_{j-1}^+(u)=\Gamma^+(u_{j-2})$, the arc
$(u_{j-2},u_{j-1})$ satisfies, by (1),
\(
\Gamma^+(u_{j-1})=\Gamma_2^+(u_{j-2}).
\)
Hence
\[
\Gamma_j^+(u)=\Gamma_2^+(u_{j-2})
=\bigcup_{z\in\Gamma_{j-1}^+(u)}\Gamma^+(z).
\]
For each $z\in\Gamma_{j-1}^+(u)$, again by (1) applied to the arc
$(u_{j-2},z)$, we have
\[
\Gamma^+(z)=\Gamma_2^+(u_{j-2})=\Gamma_j^+(u).
\]
Therefore every vertex of $\Gamma_{j-1}^+(u)$ points to every vertex of
$\Gamma_j^+(u)$. Since $|\Gamma_j^+(u)|=m$ and every vertex has in-valency
$m$, each $x\in\Gamma_j^+(u)$ has in-neighbor set exactly
$\Gamma_{j-1}^+(u)$. Thus the claim follows.

\medskip

\noindent\textbf{Claim 2.} $t\neq d$.

\emph{Proof of Claim 2.}
Suppose that $t=d$. Then
\( \Gamma_d^+(u)=\Gamma^+(u_{d-1}). \)
Let $x\in\Gamma_d^+(u)$. By Claim 1,
\(
\Gamma^-(x)=\Gamma_{d-1}^+(u).
\)
Hence $x$ cannot point to any vertex in $\Gamma_j^+(u)$ for $2\le j\le d$,
because those vertices have in-neighbor set exactly $\Gamma_{j-1}^+(u)$,
which does not contain $x$.
Thus every out-neighbor of $x$ lies in $\{u\}\cup\Gamma_1^+(u)$.

Suppose first that $x\to u$. By (1), applied to the arc $(x,u)$,
\(\Gamma^+(u)=\Gamma_2^+(x).
\)
If $x$ also points to some $y\in\Gamma_1^+(u)$, then since
$\Gamma^+(u)=\Gamma_1^+(u)$ and $\Gamma^+(y)=\Gamma_2^+(u)$, the set
$\Gamma_2^+(x)$ contains both $\Gamma_1^+(u)$ and $\Gamma_2^+(u)$. These two
sets are disjoint and both have size $m$, so
\( |\Gamma_2^+(x)|\ge 2m>m, \)
contradicting $|\Gamma_2^+(x)|=m$. Therefore $x$ has no out-neighbor in
$\Gamma_1^+(u)$. Since $x$ cannot point to $\Gamma_j^+(u)$ for $j\ge2$, the
only possible out-neighbor of $x$ is $u$. But $x$ has out-valency $m\ge2$,
a contradiction. Hence $x$ does not point to $u$.

Therefore every out-neighbor of $x$ lies in $\Gamma_1^+(u)$. Since
$|\Gamma_1^+(u)|=m$ and $x$ has out-valency $m$, we get
\(
\Gamma^+(x)=\Gamma_1^+(u)
\)
for every $x\in\Gamma_d^+(u)$.
Now fix $y\in\Gamma_1^+(u)$. The vertex $y$ has $u$ as an in-neighbor.
Moreover, every $x\in\Gamma_d^+(u)$ points to $y$. Thus $y$ has at least
\(
1+|\Gamma_d^+(u)|=1+m
\)
in-neighbors, contradicting the fact that the in-valency of $y$ is $m$.
Therefore $t\neq d$.

\medskip

So $t\le d-1$. Write
\[
\Gamma_i^+(u)=\{v_{i,1},v_{i,2},\dots,v_{i,m}\}
\]
for $1\le i\le t$, where $v_{i,1}=u_i$.

\medskip

\noindent\textbf{Claim 3.}
For every $i=1,\dots,m$,
\(
\Gamma^+(v_{t,i})\subseteq \Gamma_{t+1}^+(u)\cup\{u\}.
\)

\emph{Proof of Claim 3.}
First note that for every $i$,
\[
\Gamma^+(v_{t,i})=\Gamma_2^+(u_{t-1}),
\]
because $(u_{t-1},v_{t,i})$ is an arc and we may apply (1). Thus all the
sets $\Gamma^+(v_{t,i})$ are equal.

Let $y\in\Gamma_j^+(u)$ for some $2\le j\le t$. By Claim 1,
\( \Gamma^-(y)=\Gamma_{j-1}^+(u).
\)
Since $v_{t,i}\in\Gamma_t^+(u)$ and $\Gamma_t^+(u)\not\subseteq
\Gamma_{j-1}^+(u)$ for $j\le t$, we have $v_{t,i}\not\to y$. Hence
\[
\Gamma^+(v_{t,i})\cap\left(\bigcup_{j=2}^t\Gamma_j^+(u)\right)=\emptyset.
\]

Next we show that $v_{t,i}$ cannot point to any vertex of $\Gamma_1^+(u)$.
Suppose, on the contrary, that
\( v_{t,i}\to v \)
for some $v\in\Gamma_1^+(u)$.
If \(t=2\), then \(v\to v_{t,i}\) is an arc, because \(\Gamma_1^+(u)\to\Gamma_2^+(u)\) is complete bipartite, hence \(v_{t,i}\to v\to v_{t,i}\) is a directed \(2\)-cycle, contradicting that $\Gamma$ is oriented. If \(t\ge 3\), then since \(\Gamma_{j-1}^+(u)\to\Gamma_j^+(u)\) is complete bipartite for \(j=2,\dots,t\), we have the directed cycle
\[
v\to v_{2,1}\to v_{3,1}\to\cdots\to v_{t-1,1}\to v_{t,i}\to v
\]
of length \(t\).
 By vertex transitivity, $u$ also lies on a directed cycle of
length $t$. Let $w$ be the predecessor of $u$ on such a cycle. Along the
cycle, the distance from $u$ to $w$ is $t-1$. If $w\in\Gamma_1^+(u)$, then
$w\to u$ would form a directed $2$-cycle with $u\to w$, impossible. If
$w\in\Gamma_j^+(u)$ for some $2\le j\le t-2$, then by Claim 1, $w$ points
to every vertex of $\Gamma_{j+1}^+(u)$, so it cannot point to $u$. Thus
$w\in\Gamma_{t-1}^+(u)$. But Claim 1 also says that every vertex of
$\Gamma_{t-1}^+(u)$ points to $\Gamma_t^+(u)$, not to $u$, a contradiction.
Therefore $v_{t,i}$ has no out-neighbor in $\Gamma_1^+(u)$.

Thus every out-neighbor of $v_{t,i}$ lies in
$\Gamma_{t+1}^+(u)\cup\{u\}$, proving Claim 3.

\medskip

\noindent\textbf{Claim 4.}
Every vertex of $\Gamma_t^+(u)$ points to $u$, and
\(
\Gamma^+(v_{t,i})=\Gamma_{t+1}^+(u)\cup\{u\}.
\)

\emph{Proof of Claim 4.}
Suppose that $v_{t,i}\not\to u$. Then by Claim 3,
\(
\Gamma^+(v_{t,i})\subseteq\Gamma_{t+1}^+(u).
\)
On the other hand, every vertex of $\Gamma_{t+1}^+(u)$ is an out-neighbor of
some vertex of $\Gamma_t^+(u)$. Since all the sets $\Gamma^+(v_{t,k})$ are
equal, we get
\(
\Gamma_{t+1}^+(u)\subseteq\Gamma^+(v_{t,i}).
\)
Hence
\[
\Gamma_{t+1}^+(u)=\Gamma^+(v_{t,i}).
\]
But also
\( \Gamma^+(v_{t,i})=\Gamma_2^+(u_{t-1})=\Gamma^+(u_t),\)
so
\(
\Gamma_{t+1}^+(u)=\Gamma^+(u_t).
\)
This contradicts the maximality of $t$. Therefore
\(v_{t,i}\to u\)
for every $i=1,\dots,m$.

Since $|\Gamma_t^+(u)|=m=|\Gamma^-(u)|$, it follows that
\[
\Gamma^-(u)=\Gamma_t^+(u).
\]
By Claim 3 and the fact that $u\in\Gamma^+(v_{t,i})$, we obtain
\[
\Gamma^+(v_{t,i})=\Gamma_{t+1}^+(u)\cup\{u\}.
\]
In particular,
\(
|\Gamma_{t+1}^+(u)|=m-1.
\)

\medskip

\noindent\textbf{Claim 5.} $t+1=d$.

\emph{Proof of Claim 5.}
Take $w\in\Gamma_{t+1}^+(u)$. Since every vertex of $\Gamma_t^+(u)$ points
to $w$ and $w$ has in-valency $m$, we have
\(
\Gamma^-(w)=\Gamma_t^+(u).
\)
By (1), applied to the arc $(v_{t,1},w)$,
\( \Gamma^+(w)=\Gamma_2^+(v_{t,1}). \)
But \( \Gamma^+(v_{t,1})=\Gamma_{t+1}^+(u)\cup\{u\}.
\)
Therefore
\[
\Gamma_2^+(v_{t,1})
=
\Gamma^+(u)\cup\bigcup_{y\in\Gamma_{t+1}^+(u)}\Gamma^+(y).
\]
Since $\Gamma^+(u)=\Gamma_1^+(u)$, the set $\Gamma_1^+(u)$ is contained in
$\Gamma^+(w)$. But $|\Gamma_1^+(u)|=m$ and $w$ has out-valency $m$, so
\(
\Gamma^+(w)=\Gamma_1^+(u).
\)
If $t+1<d$, then there exists a vertex
\( z\in\Gamma_{t+2}^+(u).
\)
By definition of $\Gamma_{t+2}^+(u)$, there is some $w\in\Gamma_{t+1}^+(u)$
such that $w\to z$. But we just proved that $\Gamma^+(w)=\Gamma_1^+(u)$,
so $z\in\Gamma_1^+(u)$, contradiction. Hence $t+1=d$.

\medskip

Thus $t=d-1$. Define blocks
\[
V_0=\{u\}\cup\Gamma_d^+(u),
\]
and for $1\le i\le d-1$,
\[
V_i=\Gamma_i^+(u).
\]
We have
\(
|V_0|=1+|\Gamma_d^+(u)|=1+(m-1)=m,
\)
and for $1\le i\le d-1$,
\(
|V_i|=m.
\)
From Claims 1, 4 and 5, the arcs are precisely all arcs from $V_i$ to
$V_{i+1}$, where the indices are read modulo $d$. Therefore
\[
\Gamma\cong C_d(m,1).
\]
Setting \(r=d\), we obtain \(\Gamma\cong C_r(m,1)\) with \(r\ge 3\) and \(m\ge 2\).
This completes the proof.
\end{proof}

Let $n\ge 3$, let $T=\mathbb{Z}_n$ be the cyclic group of order $n$ under addition, and let
$S\subseteq T\setminus\{0\}$ be such that $T=\langle S\rangle$. Consider the
circulant digraph $\Gamma=\Cay(T,S)$. Denote by $u= 0_T$ the identity element of $T$.
In the additive group $T$,  $-S = \{-s : s\in S\}$.

The group of units modulo $n$, denoted $\mathbb{Z}_n^\times$
or $(\mathbb{Z}/n\mathbb{Z})^\times$, is the multiplicative group
of residue classes modulo $n$ that are coprime to $n$:
\[
\mathbb{Z}_n^\times = \{ a \in \mathbb{Z}/n\mathbb{Z} : \gcd(a,n)=1 \}.
\]
In particular, $\Aut(T)\cong \mathbb{Z}_n^\times$, which is cyclic if and only if $n=1,2,4p^k,2p^k$, $p$ is an odd prime.

\begin{lemma}\label{lem:main-1}
Let \(T\) be a cyclic group of order \(n\), and let \(S\subseteq T\) satisfy
\(T=\langle S\rangle\). Let
\(\Gamma=\operatorname{Cay}(T,S)\) be an arc-transitive circulant
digraph of order \(n\ge 3\) and valency at least \(2\). Let \((u,v)\) be an arc
with \(u=0_T\). If the right regular representation \(R(T)\) is normal in
\(G:=\Aut(\Gamma)\), then the following statements hold.
\begin{enumerate}[{\rm (1)}]
\item every element of \(S\) has order \(n\);
\item if \(\Gamma\) is \(2\)-distance-transitive, then all elements of
\(\Gamma_2^+(u)\) have the same order in \(T\);
\item \(G_u\) acts faithfully and regularly on \(S\), and \(|G_u|=|S|\);
\item if \(1\in S\), then \(S\) is a multiplicative subgroup of
\(\mathbb{Z}_n^\times\).
\end{enumerate}
\end{lemma}

\begin{proof}
Let \(G=\Aut(\Gamma)\). Since \(R(T)\unlhd G\), by Lemma \ref{cayley-normal},
\(
G=R(T)\rtimes G_u,
\)
where
\(
G_u=\Aut(T,S)=\{\varphi\in\Aut(T):\varphi(S)=S\}\le \mathbb{Z}_n^\times.
\)
Identify \(T\) with \(\mathbb{Z}_n\). Then the stabilizer of \(u=0\) is exactly
\(G_u\), acting on \(T\) by multiplication.

By arc-transitivity, \(G_u\) acts transitively on
\(S=\Gamma_1^+(u)\). Hence all elements of \(S\) have the same order in \(T\).
Since \(T\) is cyclic and \(T=\langle S\rangle\), this common order must be \(n\).
Thus \(S\subseteq \mathbb{Z}_n^\times\).

Assume further that \(\Gamma\) is \(2\)-distance-transitive. Then \(G_u\) acts
transitively on
\[
\Gamma_2^+(u)=(S+S)\setminus(S\cup\{0\}).
\]
Therefore all elements of \(\Gamma_2^+(u)\) have the same order in \(T\).

Since \(S\) generates \(T\), the action of \(G_u\) on \(S\) is faithful.
As \(G_u\) is abelian and transitive on \(S\), all point stabilizers are equal;
if one were nontrivial, it would fix every point of \(S\), contradicting
faithfulness. Hence \(G_u\) acts regularly on \(S\), and so
\(|G_u|=|S|\).

Assume now that \(1\in S\). Then the \(G_u\)-orbit of \(1\) is contained in \(S\).
Since \(G_u\) acts freely on \(\mathbb{Z}_n^\times\), we have
\[
|G_u\cdot 1|=|G_u|=|S|.
\]
Because \(1\in S\), transitivity gives \(G_u\cdot 1=S\). Thus, under the
identification \(\Aut(T)\cong\mathbb{Z}_n^\times\),
\[
S=G_u\cdot 1=G_u.
\]
Therefore \(S\) is a multiplicative subgroup of \(\mathbb{Z}_n^\times\).
\end{proof}

\begin{lemma}\label{lem:even-2}
Let $T=\mathbb Z_n$ with $n$ even and $n\ge 3$, and let
$S\subseteq T$ satisfy
\( \langle S\rangle=T.
\)
Let $\Gamma=\Cay(T,S)$ be an arc-transitive circulant digraph of
valency at least $2$. If the right regular representation $R(T)$ is
normal in $\Aut(\Gamma)$, then for every arc $(u,v)$ of $\Gamma$,
\(
\Gamma^+(u)\cap \Gamma^+(v)=\varnothing.
\)

\end{lemma}

\begin{proof}
Assume that $R(T)\trianglelefteq \Aut(\Gamma)$. Then
\(
\Aut(\Gamma)=R(T)\rtimes \Aut(T,S).
\)
Because $\Gamma$ is arc-transitive, by Lemma \ref{lem:main-1},  all elements of $S$ have the same order $n$.

Now let $(u,v)$ be an arc of $\Gamma$. Then $v-u\in S$. Since $\Gamma$ is arc-transitive, it is enough to
show that
\[
\Gamma^+(0_T)\cap \Gamma^+(v)=\varnothing
\]
for every $v\in S$.

We have
\( \Gamma^+(0_T)=S,
\)
and
\(
\Gamma^+(v)=v+S=\{v+s:s\in S\}.
\)
For every $s\in S$, both $v$ and $s$ are odd, so $v+s$ is even.
Since $n$ is even, reduction modulo $n$ preserves parity. Hence every
element of $\Gamma^+(v)$ is even, while every element of
$\Gamma^+(0_T)=S$ is odd. Therefore
\[
\Gamma^+(u)\cap \Gamma^+(v)=\varnothing.
\]
This completes the proof.
\end{proof}

The following proposition determines  all 2-distance-transitive  normal circulant digraphs of even order.

\begin{prop}\label{lem:even-3}
Let $T=\mathbb Z_n$ with   $n\geq 3$  even, and let
$S\subseteq T$ satisfy $-S\cap S=\varnothing$ and
$\langle S\rangle=T$. Let $\Gamma=\operatorname{Cay}(T,S)$ be a
$2$-distance-transitive circulant digraph of valency at least $2$.
If the right regular representation $R(T)$ is normal in
$\operatorname{Aut}(\Gamma)$, then  $\Gamma$ is isomorphic to the  directed
cycle $\overrightarrow{C_n}$.
\end{prop}

\begin{proof}
Since $\Gamma=\operatorname{Cay}(T,S)$ is $2$-distance-transitive, it
is arc-transitive. The condition   $-S\cap S=\varnothing$ implies that  $\Gamma$ is an oriented graph.  Let $u=0$ and let $(u,v)$ be an arc of $\Gamma$. Denote $G:=\operatorname{Aut}(\Gamma)$. By  the normality of  the
right regular representation $R(T)$  in
$G$ and the evenness of  $n$, Lemma \ref{lem:even-2}
yields
\( \Gamma^+(u)\cap \Gamma^+(v)=\varnothing.
\)
Consequently,   $\Gamma^+(v)\subseteq \Gamma_2^+(u)$, which implies the cardinality inequality
\[
|\Gamma_2^+(u)|\ge |\Gamma^+(v)|.
\]

As  $\Gamma$ is $2$-distance-transitive, the vertex stabilizer $G_u$ acts transitively on
$\Gamma_2^+(u)$, hence $|\Gamma_2^+(u)|$ divides $|G_u|$. By the vertex-transitivity of $\Gamma$, we have
\( |\Gamma^+(u)|=|\Gamma^+(v)|.
\)
Furthermore,   Lemma \ref{lem:main-1} gives
\( |G_u|=|S|=|\Gamma^+(u)|.
\)
Combining these qualities and inequalities, we obtain
\[
|\Gamma_2^+(u)|\ge |\Gamma^+(v)|=|S|=|G_u|.
\]
Since $|\Gamma_2^+(u)|$ divides $|G_u|$ and both quantities are positive integers, we deduce
\[
|\Gamma_2^+(u)|=|G_u|=|S|.
\]
Thus $\Gamma^+(v)$ and $\Gamma_2^+(u)$ have the same cardinality.
Since $\Gamma^+(v)\subseteq \Gamma_2^+(u)$, we obtain
\[
\Gamma^+(v)=\Gamma_2^+(u).
\]
By Lemma \ref{prop-1=2},  $\Gamma$ is either isomorphic to  $\overrightarrow{C_n}$
or to the digraph
\( \Gamma\cong C_r(m,1)
\)
for some integers $m\ge 2$ and $r\ge 3$. However, Lemma \ref{crv-not-circdig-1} confirms that  $\Gamma\cong C_r(m,1)$ is not a normal circulant digraph. This eliminates the latter case, so   $\Gamma$ is isomorphic to  $\overrightarrow{C_n}$, completing the proof.
\end{proof}

\bigskip

\bigskip

\section{ Normal $2$-distance-transitive  circulant digraphs of odd order}

In this section, we will determine  all 2-distance-transitive  normal circulant digraphs of odd order.

\begin{lemma}\label{lem:interval-stabilizer}
Let \(p>3\) be a prime, and let \(I\subseteq \mathbb F_p\) be a proper consecutive interval containing \(0\), with \(|I|\ge 3\). If \(c\in \mathbb F_p^\times\) satisfies
\( cI=I, \)
then \(c=\pm 1\).
\end{lemma}

\begin{proof}
Since \(I\) is a proper consecutive interval containing \(0\), we may write
\(
I=\{-r,-r+1,\dots,s\}
\)
for some integers \(r,s\ge 0\), with
\(
r+s+1=|I|\le p-1.
\)
If \(c=1\), there is nothing to prove. Assume \(c\neq 1\).

From \(cI=I\), we have
\[
\sum_{x\in I} cx=\sum_{x\in I}x.
\]
Hence
\[
(c-1)\sum_{x\in I}x=0,
\]
so
\[
\sum_{x\in I}x=0.
\]
But
\[
\sum_{x=-r}^{s}x=\frac{(s-r)(r+s+1)}{2}.
\]
Since \(p\) is odd and \(1\le r+s+1=|I|\le p-1\), the factor \(r+s+1\) is nonzero modulo \(p\). Therefore
\(
s-r\equiv 0\pmod p.
\)
Because \(|s-r|<p\), we get \(s=r\). Thus
\[
I=\{-r,-r+1,\dots,r\}.
\]

Now compare sums of squares. Since \(cI=I\),
\[
\sum_{x\in I}(cx)^2=\sum_{x\in I}x^2.
\]
Thus
\(
(c^2-1)\sum_{x\in I}x^2=0.
\)
If \(c^2\neq 1\), then necessarily
\(
\sum_{x\in I}x^2=0.
\)
But
\[
\sum_{x=-r}^{r}x^2
=
2\sum_{i=1}^{r}i^2
=
\frac{r(r+1)(2r+1)}{3}.
\]
Since \(|I|=2r+1\le p-1\), we have \(r\ge 1\) and \(2r+1<p\). Hence none of \(r\), \(r+1\), \(2r+1\) is divisible by \(p\). As \(p>3\), the factor \(3\) is invertible modulo \(p\), so
\(
\sum_{x\in I}x^2\neq 0
\)
in \(\mathbb F_p\). Therefore \(c^2=1\). Since \(c\neq 1\), we conclude \(c=-1\).
\end{proof}

\begin{lemma}\label{lem:two-cosets}
Let \(p\) be an odd prime, and let \(H\le \mathbb F_p^\times\) be a subgroup of odd order with \(2\in H\). If
\(
H+H=H\cup dH
\)
for some \(d\notin H\), then
\(
|H|=\frac{p-1}{2}.
\)
\end{lemma}

\begin{proof}
The assumptions force \(p>3\), since \(\mathbb F_3^\times\) has order \(2\) and its only odd-order subgroup is trivial, so it cannot contain \(2\).

Since \(H\) has odd order, \(-1\notin H\). Thus \(0\notin H+H\). From
\(
H+H=H\cup dH
\)
and \(d\notin H\), the union is disjoint, so
\[
|H+H|=|H|+|dH|=2|H|.
\]

Put
\(
A=H\cup\{0\}.
\)
Then
\[
A+A=(H+H)\cup H\cup\{0\}
=H\cup dH\cup\{0\},
\]
and the three sets \(H\), \(dH\), and \(\{0\}\) are pairwise disjoint. Hence
\[
|A+A|=2|H|+1=2|A|-1.
\]

Suppose, for contradiction, that
\(
|H|<\frac{p-1}{2}.
\)
Then, since \(|H|\) is an integer and \(p\) is odd,
\(
|H|\le \frac{p-3}{2},
\)
so
\(
|A+A|=2|H|+1\le p-2<p.
\)
Thus \(A\) attains the Cauchy--Davenport lower bound with \(|A+A|<p\). By Vosper's theorem in \cite{Vosper-1956}, \(A\) is an arithmetic progression in \(\mathbb F_p\).

Since \(0\in A\), we may write
\(
A=\lambda I
\)
for some \(\lambda\in \mathbb F_p^\times\) and some consecutive interval \(I\subseteq \mathbb F_p\) containing \(0\). Also,
\(
|I|=|A|=|H|+1\ge 4,
\)
because \(H\) is an odd-order subgroup containing \(2\), so \(|H|\ge 3\).

For every \(h\in H\), multiplication by \(h\) leaves \(H\) invariant, so
\(
hA=h(H\cup\{0\})=hH\cup\{0\}=H\cup\{0\}=A.
\)
Since \(A=\lambda I\) and multiplication in \(\mathbb F_p\) is commutative,
\(
h(\lambda I)=\lambda I
\)
implies
\(
hI=I.
\)
By Lemma~\ref{lem:interval-stabilizer}, this forces \(h=\pm 1\). But \(-1\notin H\), so \(h=1\) for every \(h\in H\). Hence \(H=\{1\}\), contradicting \(2\in H\).

Therefore
\[
|H|\ge \frac{p-1}{2}.
\]
Since \(H\) is a proper subgroup of \(\mathbb F_p^\times\), we also have
\[
|H|\le \frac{p-1}{2}.
\]
Thus
\(
|H|=\frac{p-1}{2}.
\)
\end{proof}

\begin{lemma}\label{lem:projection-classification}
Let \(p\) be an odd prime, and let \(H\le \mathbb{F}_p^\times\) be a subgroup with \(2\in H\).
If \(H+H\) is the union of exactly two distinct \(H\)-orbits under multiplication, then either
\(
H=\mathbb{F}_p^\times,
\)
or
\(
H+H=H\cup aH
\)
for some \(a\notin H\), and \(|H|=(p-1)/2\).
\end{lemma}

\begin{proof}
Since \(2\in H\), we have \(1+1=2\in H+H\). Because \(H+H\) is invariant under multiplication by \(H\), the whole orbit \(H\) is contained in \(H+H\). Thus
\[
H+H=H\cup O
\]
for another \(H\)-orbit \(O\).

If \(O=\{0\}\), then \(H+H=H\cup\{0\}\). If \(H\neq \mathbb{F}_p^\times\), then by Cauchy--Davenport,
\(
|H+H|\ge \min\{p,2|H|-1\}.
\)
But \(|H+H|=|H|+1\). If \(|H|+1<p\), then
\(
|H|+1\ge 2|H|-1,
\)
so \(|H|\le 2\). Since \(2\in H\) and \(2\neq 1\), we must have \(|H|=2\). Then \(H=\{1,2\}\), but the only subgroup of order \(2\) is \(\{1,-1\}\), so \(p=3\) and \(H=\mathbb{F}_p^\times\), contrary to assumption. Therefore \(|H|+1=p\), so \(|H|=p-1\), and again \(H=\mathbb{F}_p^\times\). Hence in this case \(H=\mathbb{F}_p^\times\).

If \(O\neq\{0\}\), then \(O=aH\) for some \(a\in \mathbb{F}_p^\times\setminus H\). Hence
\[
H+H=H\cup aH.
\]
In particular \(0\notin H+H\), so \(-1\notin H\), and therefore \(|H|\) is odd. By Lemma~\ref{lem:two-cosets},
\(
|H|=\frac{p-1}{2}.
\)
\end{proof}

\begin{lemma}\label{2notin-s}
Let \(p,q\) be distinct odd primes. There does not exist a subset \(S\subseteq \mathbb{Z}_{pq}\) such that:
\begin{itemize}
    \item \(-S\cap S=\varnothing\) and \(\langle S\rangle=\mathbb{Z}_{pq}\);
        \item \(1,2\in S\);
    \item \(\Gamma=\operatorname{Cay}(\mathbb{Z}_{pq},S)\) is a normal circulant digraph;
    \item \(\Gamma\) is \(2\)-distance-transitive with valency at least \(2\).

\end{itemize}
\end{lemma}

\begin{proof}
Suppose such a digraph \(\Gamma\) exists. Let \(K=\operatorname{Aut}(\Gamma)_0\) be the stabilizer of the vertex \(0\). Since \(\Gamma\) is  normal, by Lemma \ref{cayley-normal},
\(
\Aut(\Gamma)=R(T)\rtimes K,\)
where
\(
K=\Aut(T,S)=\{\varphi\in\Aut(T):\varphi(S)=S\}.
\)
Because \(\Gamma\) is arc-transitive and  \(1\in S\), it follows from   Lemma \ref{lem:main-1} that
\(
S=K\cdot 1=K
\)
 is a multiplicative subgroup of \(\mathbb{Z}_{pq}^\times\). Moreover, \(2\in K\) and \(|K|\ge 2\).
The condition \(-S\cap S=\varnothing\) gives \(-1\notin K\), and so   \(0\notin K+K\).

Since \(\Gamma\) is \(2\)-distance-transitive, \(K\) acts transitively on \(\Gamma_2^+(0)\). Also,
\(
\Gamma_2^+(0)=(K+K)\setminus K.
\)
Thus there is a single \(K\)-orbit \(R\) such that
\(
K+K=K\cup R,
\)
and \(R\cap K=\varnothing\).

Let
\[
H_p=\pi_p(K)\le \mathbb{F}_p^\times,\qquad H_q=\pi_q(K)\le \mathbb{F}_q^\times,
\]
where \(\pi_p,\pi_q\) are the reductions modulo \(p\) and \(q\), respectively. Projecting the equality \(K+K=K\cup R\) gives
\[
H_p+H_p=H_p\cup \pi_p(R),\qquad H_q+H_q=H_q\cup \pi_q(R).
\]
Since \(R\) is a \(K\)-orbit, \(\pi_p(R)\) is an \(H_p\)-orbit and \(\pi_q(R)\) is an \(H_q\)-orbit.

We claim that \(\pi_p(R)\ne H_p\) and \(\pi_q(R)\ne H_q\). Suppose, for instance, that \(\pi_p(R)=H_p\). Then
\(
H_p+H_p=H_p.
\)
As \(H_p\cup\{0\}\) is closed under addition and multiplication, contains \(0\) and \(1\), and every nonzero element has a multiplicative inverse in \(H_p\), it is a subfield of \(\mathbb{F}_p\). Hence \(H_p\cup\{0\}=\mathbb{F}_p\), so \(H_p=\mathbb{F}_p^\times\). But then
\[
H_p+H_p=\mathbb{F}_p^\times+\mathbb{F}_p^\times=\mathbb{F}_p\ne \mathbb{F}_p^\times=H_p,
\]
a contradiction. Therefore \(\pi_p(R)\ne H_p\). The same argument gives \(\pi_q(R)\ne H_q\).

Thus both \(H_p+H_p\) and \(H_q+H_q\) are unions of exactly two distinct \(H_p\)-orbits, respectively two distinct \(H_q\)-orbits. Since \(2\in H_p,H_q\), Lemma~\ref{lem:projection-classification} applies.

For each of \(H_p,H_q\), there are two possibilities: either it is the full multiplicative group, or it has index \(2\). We consider the three cases.

\medskip
\noindent\textbf{Case 1: \(H_p=\mathbb{F}_p^\times\) and \(H_q=\mathbb{F}_q^\times\).}
Then
\[
H_p+H_p=\mathbb{F}_p=H_p\cup\{0\},\qquad H_q+H_q=\mathbb{F}_q=H_q\cup\{0\}.
\]
From
\(
H_p+H_p=H_p\cup\pi_p(R)
\)
we get \(\pi_p(R)=\{0\}\). Similarly \(\pi_q(R)=\{0\}\). Thus every element of \(R\) is divisible by both \(p\) and \(q\), hence is \(0\) in \(\mathbb{Z}_{pq}\). But \(0\notin R\), a contradiction.

\medskip
\noindent\textbf{Case 2: one projection is full and the other has index \(2\).}
Without loss of generality, assume
\[
H_p=\mathbb{F}_p^\times,\qquad |H_q|=\frac{q-1}{2}.
\]
Then \(H_q\) has odd order, and \(H_q+H_q=H_q\cup a_qH_q\) for some \(a_q\notin H_q\). Since
\(
H_p+H_p=\mathbb{F}_p=H_p\cup\{0\},
\)
we must have \(\pi_p(R)=\{0\}\). Hence \(R\subseteq p\mathbb{Z}_{pq}\). On the other hand,
\(
H_q+H_q=H_q\cup\pi_q(R),
\)
so \(\pi_q(R)=a_qH_q\). Thus every element of \(R\) has \(p\)-component \(0\) and \(q\)-component in \(a_qH_q\).

Now regard \(K\) as a subgroup of \(H_p\times H_q\) with surjective projections. For \(x\in H_p\), let
\[
Y_x=\{y\in H_q:(x,y)\in K\}.
\]
By Goursat's Lemma (see \cite{Lang-2002}), each \(Y_x\) is a coset of a fixed subgroup \(N_q\le H_q\). Let \(n_q=|N_q|\). Since \(K+K=K\cup R\) and \(\pi_p(R)=\{0\}\), for any \(x_1,x_2\in H_p\) with \(x_1+x_2\in H_p\), we have
\(
Y_{x_1}+Y_{x_2}\subseteq Y_{x_1+x_2}.
\)
In particular, for any \(x\in H_p\),
\(
Y_{x/2}+Y_{x/2}\subseteq Y_x.
\)
If \(n_q\ge 2\), then by Cauchy--Davenport,
\(
|Y_{x/2}+Y_{x/2}|\ge \min\{q,2n_q-1\}.
\)
Since \(n_q\le |H_q|=(q-1)/2\), we have \(2n_q-1\le q-2<q\), so
\[
|Y_{x/2}+Y_{x/2}|\ge 2n_q-1.
\]
But \(|Y_x|=n_q\), hence
\[
n_q\ge 2n_q-1,
\]
so \(n_q\le 1\). Therefore \(n_q=1\). Thus every fiber \(Y_x\) is a singleton,  write \(Y_x=\{\phi(x)\}\).

Since \(K\) is a subgroup, the map \(\phi:H_p\to H_q\) is a multiplicative homomorphism. Moreover, the above inclusion becomes
\(
\phi(x_1)+\phi(x_2)=\phi(x_1+x_2)
\)
whenever \(x_1,x_2,x_1+x_2\in H_p\).

Because \(1,2\in K\), we have
\[
\phi(1)=1,\qquad \phi(2)=2.
\]
Also \(-1\in H_p=\mathbb{F}_p^\times\). Since \(\phi\) is multiplicative,
\(
\phi(-1)^2=\phi((-1)^2)=\phi(1)=1.
\)
As \(H_q\) has odd order, the only solution to \(y^2=1\) in \(H_q\) is \(y=1\). Hence \(\phi(-1)=1\). Now choose \(x_1=-1\), \(x_2=2\). Then \(x_1+x_2=1\in H_p\), so
\(
\phi(-1)+\phi(2)=\phi(1).
\)
Thus
\(
1+2=1\pmod q,
\)
which gives \(q=2\), impossible because \(q\) is an odd prime.

\medskip
\noindent\textbf{Case 3: both \(H_p\) and \(H_q\) have index \(2\).}
By Lemma~\ref{lem:projection-classification},
\[
|H_p|=\frac{p-1}{2},\qquad |H_q|=\frac{q-1}{2},
\]
and
\[
H_p+H_p=H_p\cup a_pH_p,\qquad H_q+H_q=H_q\cup a_qH_q
\]
for some \(a_p\notin H_p\), \(a_q\notin H_q\). Hence
\[
\pi_p(R)=a_pH_p,\qquad \pi_q(R)=a_qH_q.
\]
In particular, every element of \(R\) has nonzero reduction modulo both \(p\) and \(q\), so \(R\subseteq\U\). Since \(R\) is a single \(K\)-orbit, we may choose \(c\in R\subseteq\U\), and then
\(
R=cK.
\)
Therefore
\(
K+K=K\cup cK.
\)

Again view \(K\) as a subgroup of \(H_p\times H_q\) with surjective projections. By Goursat's Lemma, for each \(x\in H_p\), the fiber
\[
Y_x=\{y\in H_q:(x,y)\in K\}
\]
is a coset of a fixed subgroup \(N_q\le H_q\). Let \(n_q=|N_q|\). Since
\(
H_p+H_p=H_p\cup a_pH_p=\mathbb{F}_p^\times,
\)
for \(1\in H_p\) there exist \(x_1,x_2\in H_p\) such that \(x_1+x_2=1\). The fiber of \(K+K\) over \(1\) contains \(Y_{x_1}+Y_{x_2}\). Because \(cK\) has \(p\)-component in \(c_pH_p=a_pH_p\), it does not contribute to the fiber over \(1\in H_p\). Hence the fiber of \(K+K\) over \(1\) is exactly \(Y_1\), which has size \(n_q\). Therefore
\[
|Y_{x_1}+Y_{x_2}|\le n_q.
\]
On the other hand, by Cauchy--Davenport,
\(
|Y_{x_1}+Y_{x_2}|\ge \min\{q,2n_q-1\}.
\)
Since \(n_q\le |H_q|=(q-1)/2\), we have \(2n_q-1\le q-2<q\), so
\[
|Y_{x_1}+Y_{x_2}|\ge 2n_q-1.
\]
Thus
\[
2n_q-1\le n_q,
\]
so \(n_q=1\). By symmetry, the analogous subgroup \(N_p\le H_p\) also has order \(1\).

By Goursat's Lemma (see \cite{Lang-2002}), the quotients satisfy
\(
H_p/N_p\cong H_q/N_q.
\)
Since \(N_p\) and \(N_q\) are trivial, we get
\(
H_p\cong H_q.
\)
Hence
\(
|H_p|=|H_q|,
\)
so
\(
\frac{p-1}{2}=\frac{q-1}{2},
\)
which implies \(p=q\), contradicting the hypothesis that \(p\) and \(q\) are distinct.

\medskip
All three cases lead to contradictions. Therefore no such subset \(S\) exists.
\end{proof}


\begin{lemma}\label{lem:finite-field}
Let \(p\) be an odd prime and let \(H\le \mathbb F_p^\times\) a subgroup of   odd order. If
\(
H+H=H\cup 2H,
\)
then
\(
p\equiv 3\pmod 4\) and \(|H|=\frac{p-1}{2}.
\)
\end{lemma}

\begin{proof}
First observe that \(2\notin H\). Indeed, if \(2\in H\), then \(2H=H\), as \(H\) is a subgroup of \( \mathbb F_p^\times\), so the equality gives \(H+H=H\), a contradiction. Since
\(1\in H\), this would imply that \(H\) is closed under addition. But then the additive subgroup generated by \(1\), namely \(\mathbb F_p\), would be contained in \(H\cup\{0\}\), contradicting \(H\le \mathbb F_p^\times\). Hence
\[
H\cap 2H=\varnothing.
\]

Let \(n=|H|\). Since \(H\) has odd order, \(-1\notin H\). Thus \(0\notin H+H\): if \(h_1+h_2=0\), then \(h_2=-h_1\), forcing \(-1=h_2h_1^{-1}\in H\), a contradiction. Therefore
\[
H+H\subseteq \mathbb F_p^\times.
\]
From \(H+H=H\cup 2H\) we get
\(
|H+H|=2n,
\)
so \(2n\le p-1\).

If \(n=1\), then \(H=\{1\}\), and
\(
H+H=\{2\}\), \(H\cup 2H=\{1,2\},
\)
a contradiction. Hence \(n\ge 3\).

Set
\(
A=H\cup\{0\}.
\)
Then
\(
A+A=(H+H)\cup H\cup\{0\}
=H\cup 2H\cup\{0\}.
\)
Since \(H\), \(2H\), and \(\{0\}\) are pairwise disjoint,
\(
|A+A|=2n+1.
\)
Also
\(
|A|+|A|-1=(n+1)+(n+1)-1=2n+1.
\)
Thus
\(
|A+A|=|A|+|A|-1.
\)

Suppose for contradiction that
\[
n<\frac{p-1}{2}.
\]
Because \(n\) is odd, this implies
\(
n\le \frac{p-3}{2}.
\)
Then
\(
|A+A|=2n+1\le p-2.
\)
By Vosper's theorem in \cite{Vosper-1956}, \(A\) is an arithmetic progression in \(\mathbb F_p\).

Since \(0\in A\) and \(|A|=n+1\le (p-1)/2\), we may write
\[
A=dI,
\]
where \(d\in\mathbb F_p^\times\) and \(I\) is an interval of consecutive integers of length \(n+1\) containing \(0\). In particular,
\[
I-I=\{-(n),-(n-1),\dots,n-1,n\}.
\]
For every \(h\in H\), the fact \(H\le \mathbb F_p^\times\) gives
\(
hA=h(H\cup\{0\})=H\cup\{0\}=A.
\)
Thus
\(
h(dI)=dI,
\)
so
\(
hI=I.
\)
Consequently,
\(
h(I-I)=I-I.
\)
Equivalently, if
\(
D=\{\overline{k}\in\mathbb F_p: -n\le k\le n\},
\)
then
\(
hD=D
\)
for every \(h\in H\).

We now use the following elementary fact.

\medskip
\noindent\textbf{Claim.}
Let \(n\ge 1\) and suppose \(2n+1<p\). Put
\(
D=\{\overline{k}\in\mathbb F_p: -n\le k\le n\}.
\)
If \(a\in\mathbb F_p^\times\) satisfies \(aD=D\), then \(a=\pm 1\).

Proof:  Since \(p>2n+1\), the residues \(\overline{-n},\overline{-n+1},\dots,\overline{n}\) are pairwise distinct in \(\mathbb F_p\). Thus
\(
D=\{\overline{-n},\overline{-n+1},\dots,\overline{n}\}
\)
has exactly \(2n+1\) elements.

Consider the sum
\(
S=\sum_{x\in D} x^2 \in \mathbb F_p.
\)
Using the integer representatives, we have
\[
S=\sum_{k=-n}^{n} \overline{k}^2
=\sum_{k=-n}^{n} k^2 \pmod p
=2\sum_{k=1}^{n} k^2
=\frac{n(n+1)(2n+1)}{3} \in \mathbb F_p.
\]
Here the fraction means multiplication by the inverse of \(3\) in \(\mathbb F_p\).

We claim that \(S\neq 0\) in \(\mathbb F_p\). Indeed, since \(2n+1<p\), none of the integers \(n\), \(n+1\), or \(2n+1\) is divisible by \(p\). Also \(p\neq 3\), because \(p>2n+1\ge 3\), so \(3\) is invertible modulo \(p\). Therefore
\(
n(n+1)(2n+1)\not\equiv 0 \pmod p,
\)
and hence \(S\neq 0\) in \(\mathbb F_p\).

Now assume \(aD=D\). Then multiplication by \(a\) gives a bijection from \(D\) to itself. Therefore
\[
\sum_{x\in D} (ax)^2 = \sum_{y\in D} y^2 = S.
\]
On the other hand,
\[
\sum_{x\in D} (ax)^2
= a^2 \sum_{x\in D} x^2
= a^2 S.
\]
Thus
\(
a^2 S = S.
\)
Since \(S\neq 0\) in the field \(\mathbb F_p\), we obtain
\(
a^2 = 1.
\)
Hence
\(
(a-1)(a+1)=0,
\)
so \(a=1\) or \(a=-1\). This proves the claim.

\medskip
Returning to the proof of the lemma, we have
\(
2n+1\le p-2<p,
\)
so the claim applies. Therefore every \(h\in H\) satisfies
\(
h=\pm 1.
\)
Since \(H\) has odd order, \(-1\notin H\). Hence \(h=1\) for every \(h\in H\), forcing \(H=\{1\}\), contradicting \(n\ge 3\).

Therefore the assumption \(n<(p-1)/2\) is impossible. Hence
\(
|H|=n=\frac{p-1}{2}.
\)
Since \(|H|\) is odd, this forces
\[
\frac{p-1}{2}\equiv 1\pmod 2,
\]
so
\[
p\equiv 3\pmod 4.
\]
This completes the proof.
\end{proof}

\begin{lemma}\label{2notindistance2}
Let \(p,q\) be distinct odd primes and let \(T=\mathbb Z_{pq}\). There does not exist a subset \(S\subseteq T\) such that
\begin{itemize}
    \item \(-S\cap S=\varnothing\) and \(\langle S\rangle=T\);
    \item \(\Gamma=\operatorname{Cay}(T,S)\) is a normal circulant digraph;
    \item \(\Gamma\) is \(2\)-distance-transitive and has valency at least \(2\);
    \item \(1\in S\) and \(2\in\Gamma_2^+(0)\);
    \item \(|S|\) is odd.
\end{itemize}
\end{lemma}

\begin{proof}
Suppose such a digraph \(\Gamma\) exists.

Let \(K\) be the stabilizer of the vertex \(0\) in \(\operatorname{Aut}(\Gamma)\). Since \(\Gamma\) is a normal circulant digraph,
it follows from  Lemmas \ref{cayley-normal} and \ref{lem:main-1} that
\(
\Aut(\Gamma)=R(T)\rtimes K,\)
where
\(
K=\Aut(T,S)=\{\varphi\in\Aut(T):\varphi(S)=S\}
\) and    \( S=K\)
 is a multiplicative subgroup of \(\mathbb{Z}_{pq}^\times\).
From \(-S\cap S=\varnothing\), we have
\(
-1\notin K.
\)

Since \(2\in\Gamma_2^+(0)\) and \(\Gamma\) is \(2\)-distance-transitive, \(\Gamma_2^+(0)\) is a single \(K\)-orbit. Thus
\[
\Gamma_2^+(0)=2K.
\]
Also
\(
\Gamma_2^+(0)=(K+K)\setminus K.
\)
Therefore
\(
(K+K)\setminus K=2K.
\)

We claim that \(K\subseteq K+K\). Suppose not. Since \(K+K\) is invariant under multiplication by \(K\), we would have
\(
K\cap(K+K)=\varnothing.
\)
Then
\(
K+K=(K+K)\setminus K=2K.
\)
In particular, \(|K+K|=|K|\). Since \(1\in K\),
\(
1+K\subseteq K+K
\)
and \(|1+K|=|K|\), so
\(
K+K=1+K.
\)
Thus for every \(k\in K\),
\[
K+(k-1)=K.
\]
Let \(A\) be the additive subgroup generated by
\(
\{k-1:k\in K\}.
\)
Then \(K+A=K\). Since \(1\in K\), we have
\[
1+A\subseteq K\subseteq \mathbb{Z}_{pq}^\times.
\]
If \(A\neq\{0\}\), then \(A\) is one of
\[
p\mathbb Z_{pq},\qquad q\mathbb Z_{pq},\qquad \mathbb Z_{pq}.
\]
In each case there exists \(a\in A\) such that \(1+a\) is not a unit modulo \(pq\), contradicting \(1+a\in \mathbb{Z}_{pq}^\times\). Hence \(A=\{0\}\), forcing \(K=\{1\}\), which contradicts valency at least \(2\). Therefore
\[
K\subseteq K+K.
\]
Consequently,
\[
K+K=K\cup\Gamma_2^+(0)=K\cup 2K.
\]
Since \(2\notin K\), the cosets \(K\) and \(2K\) are disjoint, so
\[
K+K=K\sqcup 2K. \tag{1}
\]

Now project modulo \(p\) and modulo \(q\). Let
\[
H_p=\pi_p(K)\le \mathbb F_p^\times,\qquad H_q=\pi_q(K)\le \mathbb F_q^\times.
\]
Projecting (1) gives
\[
H_p+H_p=H_p\cup 2H_p, \tag{2}
\]
\[
H_q+H_q=H_q\cup 2H_q. \tag{3}
\]

We next show that \(|K|\) is odd, so \(H_p\) and \(H_q\) have odd order. From (2), the right-hand side is contained in \(\mathbb F_p^\times\), so it does not contain \(0\). If \(-1\in H_p\), then
\(
0=1+(-1)\in H_p+H_p,
\)
a contradiction. Thus \(-1\notin H_p\). Similarly,
\(
-1\notin H_q.
\)
If \(K\) contained an element of order \(2\), then under
\(
\mathbb{Z}_{pq}^\times\cong \mathbb{Z}_{p}^\times\times \mathbb{Z}_{q}^\times,
\)
that element would have both coordinates in \(\{\pm1\}\). It cannot be \((-1,-1)\), since \(-1\notin K\). Hence exactly one coordinate is \(-1\), giving \(-1\in H_p\) or \(-1\in H_q\), contradiction. Therefore \(K\) has no element of order \(2\), and so   \(|K|\) is odd. Hence \(H_p\) and \(H_q\) have odd order.

By Lemma~\ref{lem:finite-field} applied to \(H_p\) and \(H_q\), we obtain
\[
p\equiv 3\pmod 4,\qquad q\equiv 3\pmod 4,\qquad
|H_p|=\frac{p-1}{2},\qquad |H_q|=\frac{q-1}{2}. \tag{4}
\]
Also, since \(2\notin H_p\) and \(|H_p|=(p-1)/2\), we have
\(
\mathbb F_p^\times=H_p\sqcup 2H_p.
\)
Similarly,
\(
\mathbb F_q^\times=H_q\sqcup 2H_q.
\)

Now regard \(K\) as a subgroup of
\(
H_p\times H_q.
\)
Since both projections are surjective, by Goursat's Lemma there are subgroups
\[
N_p\le H_p,\qquad N_q\le H_q
\]
such that for every fixed \(x\in H_p\), the fiber
\(
Y_x=\{y\in H_q:(x,y)\in K\}
\)
is a coset of \(N_q\), and for every fixed \(y\in H_q\), the fiber
\(
X_y=\{x\in H_p:(x,y)\in K\}
\)
is a coset of \(N_p\).

Because \(K+K=K\sqcup 2K\), for every \(x\in\mathbb F_p^\times=H_p\sqcup 2H_p\), the fiber of \(K+K\) over \(x\) has size exactly
\(
|N_q|.
\)
Indeed, if \(x\in H_p\), only \(K\) contributes; if \(x\in 2H_p\), only \(2K\) contributes.

Fix \(x\in\mathbb F_p^\times\). Since
\(
\mathbb F_p^\times=H_p+H_p,
\)
there exist \(x_1,x_2\in H_p\) such that
\(
x_1+x_2=x.
\)
Let \(Y_{x_1},Y_{x_2}\subseteq H_q\) be the corresponding fibers of \(K\). Then the fiber of \(K+K\) over \(x\) contains
\(
Y_{x_1}+Y_{x_2}.
\)
Both \(Y_{x_1}\) and \(Y_{x_2}\) are cosets of \(N_q\), so
\[
|Y_{x_1}|=|Y_{x_2}|=|N_q|.
\]
By the Cauchy--Davenport theorem in \(\mathbb F_q\),
\[
|Y_{x_1}+Y_{x_2}|\ge \min(q,\,2|N_q|-1).
\]
Since
\(
|N_q|\le |H_q|=\frac{q-1}{2},
\)
we have
\(
2|N_q|-1\le q-2<q.
\)
Therefore
\(
|Y_{x_1}+Y_{x_2}|\ge 2|N_q|-1.
\)
But the full fiber of \(K+K\) over \(x\) has size \(|N_q|\). Hence
\(
|N_q|\ge 2|N_q|-1,
\)
which implies
\(
|N_q|\le 1.
\)
Thus
\[
|N_q|=1.
\]

By the same argument, using fibers over \(\mathbb F_q^\times=H_q\sqcup 2H_q\), we obtain
\[
|N_p|=1.
\]

Therefore the common quotient from Goursat's Lemma satisfies
\(
|H_p/N_p|=|H_q/N_q|.
\)
Since \(N_p\) and \(N_q\) are trivial, this gives
\(
|H_p|=|H_q|.
\)
Using (4), we get
\(
\frac{p-1}{2}=\frac{q-1}{2},
\)
so
\(p=q.\)
This contradicts the assumption that \(p\) and \(q\) are distinct.

Hence no such digraph \(\Gamma\) exists.
\end{proof}

\begin{lemma}\label{pq-even}
Let \(p,q\) be distinct odd primes. There is no subset
\(S\subseteq \mathbb Z_{pq}\) satisfying all of the following:
\begin{itemize}
    \item \(-S\cap S=\varnothing\) and \(\langle S\rangle=\mathbb Z_{pq}\);
    \item \(\Gamma=\operatorname{Cay}(\mathbb Z_{pq},S)\) is a normal circulant digraph;
    \item \(\Gamma\) is \(2\)-distance-transitive and has valency at least \(2\);
    \item \(1\in S\) and \(2\in \Gamma_2^+(0)\);
    \item \(|S|\) is even.
\end{itemize}
\end{lemma}

\begin{proof}
Assume, for contradiction, that such a subset \(S\) exists. Let
\[
\Gamma=\operatorname{Cay}(\mathbb Z_n,S), n=pq.
\]
For \(i\ge 0\), write
\(
\Gamma_i^+(0)=\{x\in \mathbb Z_n : d(0,x)=i\}.
\)
Then \(\Gamma_1^+(0)=S\).

Since \(\Gamma\) is a normal circulant digraph,
it follows from  Lemmas \ref{cayley-normal} and \ref{lem:main-1} that
\(
\Aut(\Gamma)=R(T)\rtimes K,\)
where
\(
K=\Aut(T,S)=\{\varphi\in\Aut(T):\varphi(S)=S\}
\) and    \( S=K\)
 is a multiplicative subgroup of \(\mathbb{Z}_{pq}^\times\).
From \(-S\cap S=\varnothing\), we have
\(
-1\notin K.
\)
 Since \(K\) is transitive on \(\Gamma_2^+(0)\) and  \(2\in \Gamma_2^+(0)\),
we have \(\Gamma_2^+(0)=K\cdot 2.\)
Because every element of \(K\) is a unit and \(2\) is a unit modulo \(n\), every element of \(\Gamma_2^+(0)\) is also a unit.

Since $S+S=S\cup \Gamma_2^+(0)$, it follows that   every sum of two elements of \(S\) is a unit modulo \(n\), and so
\[
S+S\subseteq \mathbb Z_n^\times.
\]
Equivalently, no sum of two elements of \(S\) is divisible by \(p\) or by \(q\).

For \(r\in\{p,q\}\), let \(S_r\) be the image of \(S\) under the reduction map
\(\mathbb Z_n\to \mathbb F_r\). Since \(S\) is a subgroup of \(\mathbb Z_n^\times\), its image \(S_r\) is a subgroup of \(\mathbb F_r^\times\).

We now show that
\[
-1\notin S_p \quad\text{and}\quad -1\notin S_q.
\]
Suppose, for instance, that \(-1\in S_p\). Since \(S_p\) is a nonempty subgroup of \(\mathbb F_p^\times\), for any \(a\in S_p\) we have
\(
-a=(-1)a\in S_p.
\)
Thus there exist \(s,t\in S\) such that
\[
s\equiv a\pmod p,\qquad t\equiv -a\pmod p.
\]
Then \(s+t\equiv 0\pmod p\), contradicting the fact that no sum of two elements of \(S\) is divisible by \(p\). Hence \(-1\notin S_p\). The same argument gives \(-1\notin S_q\).

Finally, since \(|K|=|S|\) is even, the finite group \(K\) has an element \(x\) of order \(2\). Thus
\(
x^2\equiv 1\pmod n.
\)
By the Chinese remainder theorem, write
\[
x\equiv \varepsilon_p\pmod p,\qquad x\equiv \varepsilon_q\pmod q,
\]
where \(\varepsilon_p,\varepsilon_q\in\{\pm1\}\). Since \(x\neq 1\), at least one of \(\varepsilon_p,\varepsilon_q\) equals \(-1\). If \(\varepsilon_p=-1\), then \(-1\in S_p\), a contradiction. If \(\varepsilon_q=-1\), then \(-1\in S_q\), again a contradiction.

Therefore no such subset \(S\) exists.
\end{proof}

\begin{lemma}\label{1-in-s}
Let \(T=\mathbb{Z}_n\) and let \(S\subseteq T\) satisfy
\( -S\cap S=\varnothing, \langle S\rangle=T. \)
Let \(\Gamma=\Cay(T,S)\). Then for each   unit \(a\) in \(T\),
\( \Cay(T,S)\cong \Cay(T,a^{-1}S).
\)
In particular, if  \(S\) contains a unit, then we may assume  that
\( 1\in S.
\)

\end{lemma}

\begin{proof}
Since \(a\) is a unit in \(\mathbb{Z}_n\), we have \(a\in(\mathbb{Z}_n)^\times\). Define the map
\[
\varphi:T\to T,\qquad \varphi(x)=a^{-1}x.
\]
Then \(\varphi\) is an automorphism of the additive group \(T\), because multiplication by a unit is a bijective group homomorphism of \(\mathbb{Z}_n\).

Now recall that in the Cayley digraph \(\Cay(T,S)\), there is an arc from \(x\) to \(y\) if and only if
\(y-x\in S.
\)
Under the map \(\varphi\), an arc \((x,x+s)\) with \(s\in S\) is sent to
\( (\varphi(x),\varphi(x+s)).
\)
Since \(\varphi\) is additive,
\(
\varphi(x+s)=\varphi(x)+\varphi(s).
\)
Thus the image of the arc is
\(
(\varphi(x),\varphi(x)+\varphi(s)),
\)
which is an arc of \(\Cay(T,\varphi(S))\).

Conversely, the inverse map \(\varphi^{-1}(x)=ax\) maps arcs of \(\Cay(T,\varphi(S))\) back to arcs of \(\Cay(T,S)\). Therefore \(\varphi\) induces a digraph isomorphism
\[
\Cay(T,S)\cong \Cay(T,\varphi(S)).
\]
But
\(\varphi(S)=a^{-1}S,
\)
so
\[
\Cay(T,S)\cong \Cay(T,a^{-1}S).
\]

Finally, assume \(S\) contains a unit   \(a\). Then   we have
\( 1=a^{-1}a\in a^{-1}S.
\)
Since $(a^{-1}S)^{-1}\cap a^{-1}S=\varnothing, \langle a^{-1}S\rangle =T$,   replacing \(S\) by the isomorphic connection set \(a^{-1}S\), we may assume without loss of generality that \(1\in S\).
\end{proof}

\begin{prop}\label{pq-notexit}
Let \(T=\mathbb Z_{pq}\), where \(p,q\) are distinct odd primes. Let \(S\subseteq T\) satisfy
\(
-S\cap S=\varnothing, \langle S\rangle=T.
\)
Suppose that \(\Gamma=\operatorname{Cay}(T,S)\) is  \(2\)-distance-transitive and has valency at least \(2\).
Then $\Gamma$ is not normal.

\end{prop}
\begin{proof}
Assume for  contradiction that $\Gamma$ is  normal.
Then by   Lemmas \ref{cayley-normal} and \ref{lem:main-1},
\(
\Aut(\Gamma)=R(T)\rtimes \Aut(T,S),\)
where
\(\Aut(T,S)=\{\varphi\in\Aut(T):\varphi(S)=S\}.\)
Since \(\Gamma\) is  \(2\)-distance-transitive, $K=\Aut(T,S)$ acts transitive on $S$. Since
\(T=\mathbb Z_{pq}\) and \(\langle S\rangle = T\), it follows that
 $S$ contains units,  and by Lemma \ref{1-in-s},  \(1\in S\).
Thus   $2\in S$ or $\Gamma_2^+(0)$.

If $2\in S$, then  Lemma \ref{2notin-s} gives a contradiction.
If  $2\in \Gamma_2^+(0)$, then Lemma \ref{2notindistance2} rules out the case where $|S|$ is odd, and
 Lemma \ref{pq-even} rules out the case where  $|S|$ is even.
 In every case we obtain a contradiction.

Therefore   $\Gamma$ is not normal.
\end{proof}

\begin{prop}\label{empty-equal-1}
Let \(T=\mathbb{Z}_n\) with \(n\ge 3\), and let \(S\subseteq T\) satisfy
\(
-S\cap S=\varnothing\) and \( S\rangle=T.\)
Let \(\Gamma=\operatorname{Cay}(T,S)\) be a normal circulant digraph which is
\(2\)-distance-transitive. Assume that \(\Gamma\) has valency at least \(2\)
and diameter at least \(2\). If
\(
\Gamma^+(u)\cap \Gamma^+(v)=\varnothing
\)
for some arc \((u,v)\), then
\(
\Gamma_2^+(u)=\Gamma^+(v).
\)
Consequently,
\(
\Gamma\cong C_r(m,1)
\)
for some integers \(m\ge 2\) and \(r\ge 3\).
\end{prop}

\begin{proof}
Assume that \(\Gamma^+(u)\cap \Gamma^+(v)=\varnothing\) for some arc \((u,v)\).
Since \(\Gamma\) is vertex-transitive, we may assume \(u=0_T\). Then \(v\in S\).
Set \(A=\Aut(\Gamma)\).

Since \(\Gamma\) is a normal circulant digraph, \(R(T)\unlhd A\). Because
\(T=\langle S\rangle\), Lemma \ref{lem:main-1} implies that the action of
\(A_u\) on \(S\) is faithful and that
\(
|A_u|=|S|.
\)

Since \(\Gamma\) is \(2\)-distance-transitive, \(A_u\) acts transitively on
both \(S=\Gamma_1^+(u)\) and \(\Gamma_2^+(u)\). Hence
\[
|\Gamma_2^+(u)|\mid |A_u|=|S|.
\]

Now \(\Gamma^+(u)\cap \Gamma^+(v)=\varnothing\) implies
\(
\Gamma^+(v)\subseteq \Gamma_2^+(u).
\)
Indeed, if \(w=v+t\in \Gamma^+(v)\) with \(t\in S\), then
\(w\notin \Gamma^+(u)\) by the disjointness assumption. Also \(w\neq u\),
 otherwise \(t=-v\in -S\cap S\), contradicting \(-S\cap S=\varnothing\).
Thus \(w\) lies at distance \(2\) from \(u\), so \(w\in \Gamma_2^+(u)\).
Therefore
\[
|\Gamma_2^+(u)|\ge |\Gamma^+(v)|.
\]
By vertex-transitivity, \(|\Gamma^+(v)|=|S|\). Hence
\(
|\Gamma_2^+(u)|\ge |S|.
\)
Together with \(|\Gamma_2^+(u)|\mid |S|\), this forces
\[
|\Gamma_2^+(u)|=|S|=|\Gamma^+(v)|.
\]
Since \(\Gamma^+(v)\subseteq \Gamma_2^+(u)\), we obtain
\[
\Gamma_2^+(u)=\Gamma^+(v).
\]

Therefore, by Lemma  \ref{prop-1=2}, \(\Gamma\) is either a directed
cycle or isomorphic to \(C_r(m,1)\) for some integers \(m\ge 2\) and
\(r\ge 3\). Since \(\Gamma\) has valency at least \(2\), it is not a directed
cycle, and hence
\(
\Gamma\cong C_r(m,1).
\)
\end{proof}

\begin{prop}\label{circulant-p-1}
Let $p$ be an odd prime, and let $\Gamma = \Cay(T, S)$ be a $2$-distance-transitive circulant digraph, where $T \cong \mathbb{Z}_p$, $-S \cap S = \varnothing$, and $\langle S \rangle = T$. Assume that the right regular representation $R(T)$ is normal in $ \Aut(\Gamma)$. If $|S| \geq 2$, then $\Gamma$ is isomorphic to the Paley tournament of order $p$ with $p \equiv 3 \pmod{4}$.
\end{prop}

\begin{proof}
Identify $T$ with the additive group of the finite field $\mathbb{F}_p$. Let
$(u, v)$ be an arc of $\Gamma$ with $u=0$.
If  $\Gamma^+(u) \cap \Gamma^+(v) = \varnothing$, then by Proposition \ref{empty-equal-1},
\( \Gamma \cong C_r(m,1)
\)
for some integers \(m\ge 2\) and \(r\ge 3\) with order $mr$, contradicting that $\Gamma$ has prime $p$ vertices.
Thus   \[\Gamma^+(u) \cap \Gamma^+(v) \neq \varnothing.\]

Since $R(T) \trianglelefteq G$, \( G = R(T) \rtimes H,\)
where $H =\Aut(T,S) \leq \mathbb{F}_p^\times$ is the stabilizer of the zero vertex $u=0 \in \mathbb{F}_p$, acting on $\mathbb{F}_p$ by field multiplication.
Moreover,
\(
S = H
\)
and  $|S| = |H| =: m$.

The condition $-S \cap S = \varnothing$ in the additive group $\mathbb{F}_p$, is equivalent to $H \cap (-H) = \varnothing$, which holds if and only if $-1 \notin H$. Since $\mathbb{F}_p^\times$ is cyclic of even order $p-1$, it contains a unique element of order $2$, namely $-1$. Thus $H$ contains no element of order $2$, so $m = |H|$ is odd.

We now invoke the assumption that $\Gamma^+(u) \cap \Gamma^+(v) \neq \varnothing$ for every arc $(u, v)$. By vertex-transitivity we may take $u = 0$ and $v = a \in S$. Then
\[
\Gamma^+(0) \cap \Gamma^+(a) = S \cap (a + S) \neq \varnothing.
\]
Substituting $S = aH$ gives $aH \cap (a + aH) \neq \varnothing$. Dividing by $a \neq 0$, this is equivalent to
\[
H \cap (1 + H) \neq \varnothing, \quad \text{i.e.,} \quad H \cap (H + H) \neq \varnothing.
\]
Observe that $H \cap (H + H)$ is invariant under multiplication by elements of $H$. Since $H$ is a single orbit under its own multiplication action, a nonempty invariant subset must coincide with $H$. Hence
\[
H \subseteq H + H.
\]

The set of vertices at directed distance $2$ from $0$ is
\(
\Gamma_2^{+}(0) = (S + S) \setminus S.
\)
Since $S + S = a(H + H)$, we have
\(
\Gamma_2^{+}(0) = a \cdot \bigl( (H + H) \setminus H \bigr).
\)
By $2$-distance-transitivity, $H$ acts transitively on $\Gamma_2^{+}(0)$, so $(H + H) \setminus H$ is a single $H$-orbit. As the $H$-orbits on $\mathbb{F}_p^\times$ are multiplicative cosets of $H$, there exists $c \in \mathbb{F}_p^\times \setminus H$ such that
\[
(H + H) \setminus H = cH.
\]
Combined with $H \subseteq H + H$, this yields
\[
H + H = H \cup cH.
\]
In particular, $|H + H| = 2m$. Note that $0 \notin H + H$, for otherwise there would exist $h_1, h_2 \in H$ with $h_1 + h_2 = 0$, i.e.\ $-h_1 = h_2 \in H$, contradicting $-1 \notin H$. Hence $2m \leq p - 1$.

Since $H$ is the unique subgroup of $\mathbb{F}_p^\times$ of order $m$, it consists precisely of the $m$-th roots of unity:
\[
H = \{ x \in \mathbb{F}_p^\times : x^m = 1 \}.
\]
Let $\lambda = c^m$. Since $c \notin H$, we have $\lambda \neq 1$. For any $x, y \in H$, the sum $x + y$ lies in $H \cup cH$, so
\(
(x + y)^m \in \{ 1, \lambda \}.
\)
Setting $y = 1$, we see that for every $x \in H$,
\(
(x + 1)^m - 1 = 0 \quad \text{or} \quad (x + 1)^m - \lambda = 0.
\)
Thus every element of $H$ is a root of the polynomial
\[
f(x) = \bigl( (x + 1)^m - 1 \bigr) \bigl( (x + 1)^m - \lambda \bigr).
\]
Since $H$ is exactly the set of roots of $x^m - 1$, and $x^m - 1$ has no repeated roots, $x^m - 1$ divides $f(x)$ in $\mathbb{F}_p[x]$, i.e.
\[
\bigl( (x + 1)^m - 1 \bigr) \bigl( (x + 1)^m - \lambda \bigr) \equiv 0 \pmod{x^m - 1}.
\]

We expand $(x + 1)^m$ modulo $x^m - 1$. By the binomial theorem and $x^m \equiv 1$,
\[
(x + 1)^m = \sum_{k=0}^m \binom{m}{k} x^k \equiv 2 + \sum_{k=1}^{m-1} \binom{m}{k} x^k.
\]
Define
\[
B(x) = \sum_{k=1}^{m-1} \binom{m}{k} x^k,
\]
so that $(x + 1)^m \equiv 2 + B(x) \pmod{x^m - 1}$. Substituting into the congruence gives
\(
(1 + B(x)) (2 - \lambda + B(x)) \equiv 0 \pmod{x^m - 1},
\)
or equivalently
\(
B(x)^2 + (3 - \lambda) B(x) + (2 - \lambda) \equiv 0 \pmod{x^m - 1}.
\)

We first compare constant terms. The constant term of $B(x)$ is $0$. The constant term of $B(x)^2$ arises from products $x^k \cdot x^{m-k} = x^m \equiv 1$, and equals $\sum_{k=1}^{m-1} \binom{m}{k}^2$. By Vandermonde's identity,
\[
\sum_{k=0}^m \binom{m}{k}^2 = \binom{2m}{m},
\]
so
\[
\sum_{k=1}^{m-1} \binom{m}{k}^2 = \binom{2m}{m} - 2.
\]
Equating constant terms:
\[
\binom{2m}{m} - 2 + (2 - \lambda) = 0 \quad \Longrightarrow \quad \lambda = \binom{2m}{m}.
\]

Next, we compare coefficients of $x^j$ for $1 \leq j \leq m-1$. The coefficient of $x^j$ in $B(x)^2$ comes from pairs $(k, j-k)$ and pairs $(k, j+m-k)$ (since $x^{j+m} \equiv x^j$). Using the convolution property of binomial coefficients modulo $x^m - 1$, the coefficient is
\[
\binom{2m}{j} + \binom{2m}{m-j} - 4\binom{m}{j}.
\]
Equating coefficients in the congruence yields
\[
\binom{2m}{j} + \binom{2m}{m-j} - 4\binom{m}{j} + (3 - \lambda) \binom{m}{j} = 0,
\]
which simplifies to
\[
\binom{2m}{j} + \binom{2m}{m-j} = (1 + \lambda) \binom{m}{j}. \tag{$\ast$}
\]

Now set $j = 1$ in ($\ast$):
\[
\binom{2m}{1} + \binom{2m}{m-1} = (1 + \lambda) \binom{m}{1}.
\]
That is,
\[
2m + \binom{2m}{m-1} = (1 + \lambda) m.
\]
Using the identity
\[
\binom{2m}{m-1} = \frac{m}{m+1} \binom{2m}{m} = \frac{m}{m+1} \lambda,
\]
substitute and divide by $m$:
\(
2 + \frac{\lambda}{m+1} = 1 + \lambda.
\)
Rearranging gives
\[
1 = \lambda \cdot \frac{m}{m+1} \quad \Longrightarrow \quad \lambda = \frac{m+1}{m}.
\]

Next set $j = 2$ in ($\ast$):
\[
\binom{2m}{2} + \binom{2m}{m-2} = (1 + \lambda) \binom{m}{2}.
\]
We have
\[
\binom{2m}{2} = m(2m - 1), \quad \binom{m}{2} = \frac{m(m-1)}{2},
\]
and
\[
\binom{2m}{m-2} = \frac{m(m-1)}{(m+1)(m+2)} \binom{2m}{m} = \frac{m(m-1)}{(m+1)(m+2)} \lambda.
\]
Substitute $\lambda = \frac{m+1}{m}$:
\[
\binom{2m}{m-2} = \frac{m-1}{m+2}.
\]
Substitute all terms into the $j=2$ equation:
\(
m(2m - 1) + \frac{m-1}{m+2} = \frac{2m+1}{m} \cdot \frac{m(m-1)}{2} = \frac{(2m+1)(m-1)}{2}.
\)
Multiply through by $2(m+2)$ and simplify:
\[
2m(2m-1)(m+2) + 2(m-1) = (2m+1)(m-1)(m+2).
\]
Expanding both sides and collecting terms yields
\(
2m^3 + 3m^2 + m = m(2m+1)(m+1) = 0
\)
in $\mathbb{F}_p$. Since $1 \leq m < p$, neither $m$ nor $m+1$ is divisible by $p$. Therefore
\[
2m + 1 = p \quad \Longrightarrow \quad m = \frac{p-1}{2}.
\]

Thus $H$ is the unique subgroup of $\mathbb{F}_p^\times$ of index $2$, i.e.\ the set $Q$ of nonzero quadratic residues in $\mathbb{F}_p$. Since $-1 \notin H$, the element $-1$ is a non-residue, which implies $p \equiv 3 \pmod{4}$.

Finally, $S = aH = aQ$. The map $\varphi : x \mapsto a^{-1}x$ is an automorphism of the additive group $\mathbb{F}_p$, and $\varphi(S) = Q$. Hence
\[
\Gamma = \Cay(\mathbb{F}_p, aQ) \cong \Cay(\mathbb{F}_p, Q),
\]
the Paley tournament of order $p$. This completes the proof.
\end{proof}

\begin{lemma}\label{2dtdig-quotient-nleqt}
Let $T$ be a cyclic group, and let $S\subseteq T$ satisfy  $S\cap -S=\varnothing$ and $T=\langle S\rangle$. Let $\Gamma=\operatorname{Cay}(T,S)$ be a  $2$-distance-transitive circulant digraph, and  let    $N$ be a  nontrivial proper normal  subgroup of $\Aut(\Gamma)$ with at least three orbits.   Then $N\leq T$.
\end{lemma}

\begin{proof}
The condition  $S\cap -S=\varnothing$ ensures that $\Gamma\ncong \K_{m[b]}$ for any integers $m\geq 2$ and $b\geq 3$.
Since $\Gamma=\operatorname{Cay}(T,S)$ is    $2$-distance-transitive and  $N$ is a nontrivial proper normal subgroup of $\Aut(\Gamma)$ with at least three orbits, it follows
from \cite[Lemma 5.3]{DevillersGiudiciLiPraeger2012} that
$N$ acts  semiregularly on $V(\Gamma)=T$ and is precisely  the kernel of the induced action of $\Aut(\Gamma)$ on the vertex quotient digraph $\Gamma_N$. This implies  the order relation $|V(\Gamma_N)|=|T|/|N|$.

Additionally, the group $T$ acts transitively on $V(\Gamma)$,  so
the quotient group $TN/N\cong T/(T\cap N)$ acts transitively on $V(\Gamma_N)$. As a cyclic group,    $T$ is abelian, so its quotient $T/(T\cap N)$ is also abelian. Any abelian transitive group action on a finite set is regular, and thus
 $TN/N\cong T/(T\cap N)$ acts regularly   on $V(\Gamma_N)$.  By the definition of regular group actions, we have   $| T/(T\cap N)|=|V(\Gamma_N)|$.

 Combining the above equalities yields   $| T/(T\cap N)|=|T|/|N|$, which forces  $T\cap N=N$. This inclusion immediately implies   $N\leq T$, which completes the proof.
\end{proof}

\begin{prop}\label{2dtdig-normal-primepower}
Let \(T\) be a cyclic group of odd order \(n\ge 3\) and \(S\subseteq T\) with \(-S\cap S=\varnothing\) and \(T=\langle S\rangle\). Let \(\Gamma=\operatorname{Cay}(T,S)\) be a normal \(2\)-distance-transitive circulant digraph of valency at least \(2\). Then \(n\) is a power of an odd prime \(p\) with \(p\equiv 3\pmod{4}\).
\end{prop}

\begin{proof}
Let \(G=\operatorname{Aut}(\Gamma)\). Then \(T\le G\) and \(T\) acts regularly on \(V(\Gamma)=T\).

We first consider the case where \(G\) is quasiprimitive on \(V(\Gamma)\). Since \(G\) is quasiprimitive and contains the cyclic regular subgroup \(T\),
by \cite[Theorem 1.2]{LP-circulant-2012},  $G$ is primitive on $V(\Gamma)$, and
Lemma \ref{primitive-cyclic-1} implies that \(G\) belongs to one of the cases (i)--(iv). If \(G\) lies in one of the cases (ii)--(iv), then \(G\) is \(2\)-transitive. As \(G= \operatorname{Aut}(\Gamma)\), this forces \(\Gamma\) to be a complete digraph: indeed, if \(\Gamma\) has an arc \(u\to v\), then for every ordered pair \((x,y)\) there exists \(g\in G\) with \(g(u)=x\) and \(g(v)=y\), so \(x\to y\) is an arc. Hence \(\Gamma\) would be a complete digraph, contradicting \(S\cap -S=\varnothing\) and \(n\ge 3\). Therefore \(G\) is in case (i), so \(n=p\) is a prime. By Proposition \ref{circulant-p-1}, \(p\equiv 3\pmod{4}\).

Now assume that \(G\) is not quasiprimitive on \(V(\Gamma)\). Let    $N$ be a nontrivial proper normal subgroup of $G$ with at least 2  orbits.
If $N$ has 2 orbits on $V(\Gamma)$, then as $\Gamma$ is arc-transitive, every $N$ does not contain any arc of $\Gamma$, and so $\Gamma$ is bipartite. Hence
$|V(\Gamma)|$ is even, a contradiction.  Thus $N$ has at least 3 orbits on $V(\Gamma)$, and  so  by Lemma \ref{2dtdig-quotient-nleqt}, $N\leq T$.
Let $N$ be a maximal such normal subgroup of $G$.

Since $S\cap -S=\varnothing$, we have $\Gamma\ncong \K_{m[b]}$ for any integers $m\geq 2,b\geq 3$, and so $\Gamma_N$ is not a complete digraph. By \cite[Lemma 5.3]{DevillersGiudiciLiPraeger2012}, \(N\) is semiregular on \(V(\Gamma)=T\), \(N\) is the kernel of the \(G\)-action on \(V(\Gamma_N)\), and \(\Gamma_N\) is either  a \((G/N,2)\)-distance-transitive circulant digraph or a complete digraph. In particular,
\[
|V(\Gamma_N)|=|T|/|N|.
\]
Since \(N\) is maximal among the normal subgroups with at least three orbits, the quotient \(G/N\) acts quasiprimitively or biquasiprimitively on \(V(\Gamma_N)\). If \(G/N\) were biquasiprimitive, then \(\Gamma_N\) would be bipartite and hence \(|V(\Gamma_N)|\) would be even; this would force \(n\) to be even, a contradiction. Thus \(G/N\) is quasiprimitive on \(V(\Gamma_N)\). By \cite[Theorem 1.2]{LP-circulant-2012}, every quasiprimitive group with a regular cyclic subgroup is primitive. Hence \(G/N\) is primitive on \(V(\Gamma_N)\).

Now \(T/N\) is a cyclic regular subgroup of \(G/N\). Applying Lemma \ref{primitive-cyclic-1} to the primitive group \(G/N\), we see that \(G/N\) is in one of the cases (i)--(iv). If \(G/N\) were in one of the cases (ii)--(iv), then \(G/N\) would be \(2\)-transitive, and  this  forces \(\Gamma_N\) to be a complete digraph, which is impossible. Therefore \(G/N\) must be in case (i). Hence \(|V(\Gamma_N)|=p\) is a prime, and consequently
\[
|T:N|=|T/N|=|V(\Gamma_N)|=p
\]
is a prime.

Now suppose, for a contradiction, that \(n\) has at least two distinct odd prime divisors. Let \(p,q\) be two such divisors. By the above, we may choose a maximal normal subgroup \(N\le T\) with \(|T:N|=p\). Then \(q\mid |N|\). Let \(M\le N\) be a subgroup with \(|N:M|=q\). If \(|N|=q\), then \(M=1\) and \(|T|=pq\). By Proposition \ref{pq-notexit}, \(\Gamma_M\) does not exist. Hence \(|N|\ne q\), so \(|T:M|=pq\). Since \(N\) is cyclic, \(M\) is characteristic in \(N\), and hence \(M\trianglelefteq G\). Thus \(\Gamma_M\) has order \(pq\). Moreover, \(M\unlhd G\), \(M\le N\le T\), so \(T/M\unlhd G/M\), and \(\Gamma_M\) is \(G/M\)-normal. By Proposition \ref{pq-notexit}, \(\Gamma_M\) does not exist. This contradiction shows that \(n\) is a power of a single odd prime \(p\), say \(n=p^k\).

If \(k=1\), then Proposition \ref{circulant-p-1} gives \(p\equiv 3\pmod{4}\). If \(k\ge 2\), choose \(N\le T\) with \(|T:N|=p\). Then \(\Gamma_N\) is normal and has \(p\) vertices, so Proposition \ref{circulant-p-1} again yields \(p\equiv 3\pmod{4}\). This completes the proof.
\end{proof}

\begin{lemma}\label{2d-1-orbit}
Let \(p\) be an odd prime and \(m \ge 1\). Let \(H_r\) be a multiplicative subgroup of \(\mathbb{Z}_{p^m}^\times\) of order \(r\), where \(r \mid p-1\). Suppose that
\(H_r \cap (-H_r) = \varnothing
\)
and
\(
\left| (H_r + H_r) \setminus H_r \right| = r.
\)
Then the set
\(
S := (H_r + H_r) \setminus H_r
\)
is a single orbit under the multiplicative action of \(H_r\) on \(\mathbb{Z}_{p^m}\).
\end{lemma}

\begin{proof}
Let \(H_r \le \mathbb{Z}_{p^m}^\times\) be as in the statement. Since \(r \mid p-1\), the group \(H_r\) is contained in the unique subgroup of \(\mathbb{Z}_{p^m}^\times\) of order \(p-1\), which is the Teichm\"uller lift of \(\mathbb{F}_p^\times\). In particular, the reduction map
\[
H_r \to \overline{H_r} \le \mathbb{F}_p^\times
\]
is an isomorphism.

\smallskip
\noindent
\textbf{Step 1.} We first prove that every element of \(H_r + H_r\) is a unit modulo \(p^m\).

Assume, for contradiction, that there exist \(h_1, h_2 \in H_r\) such that
\[
p \mid (h_1 + h_2).
\]
Reducing modulo \(p\), we get
\(
\overline{h_1} + \overline{h_2} = 0 \quad \text{in } \mathbb{F}_p.
\)
Thus
\(
\overline{h_1} = -\overline{h_2},
\)
which implies
\(
-1 = \overline{h_1}\,\overline{h_2}^{-1} \in \overline{H_r}.
\)
Since the reduction map is an isomorphism, we would have \(-1 \in H_r\). But this contradicts the assumption \(H_r \cap (-H_r) = \varnothing\). Therefore, for all \(h_1,h_2\in H_r\), the sum \(h_1+h_2\) is not divisible by \(p\), hence \(h_1+h_2 \in \mathbb{Z}_{p^m}^\times\).

\smallskip
\noindent
\textbf{Step 2.} The set \(S = (H_r + H_r) \setminus H_r\) is invariant under multiplication by \(H_r\).

Indeed, for any \(g \in H_r\),
\[
g(H_r + H_r) = gH_r + gH_r = H_r + H_r,
\]
and also \(gH_r = H_r\). Hence \(gS = S\).

\smallskip
\noindent
\textbf{Step 3.} The action of \(H_r\) on \(S\) is free.

Let \(g \in H_r\) and suppose \(gx = x\) for some \(x \in S\). Then
\[
(g-1)x = 0 \quad \text{in } \mathbb{Z}_{p^m}.
\]
By Step 1, \(x\) is a unit, so we can cancel \(x\) and obtain \(g-1=0\), i.e.\ \(g=1\). Hence the stabilizer of every element of \(S\) is trivial, and every orbit has size \(|H_r| = r\).

\smallskip
\noindent
\textbf{Step 4.} Since \(|S| = r\) and \(S\) is a disjoint union of orbits each of size \(r\), there can be only one such orbit. Therefore
\[
S = H_r \, x
\]
for some \(x \in S\), and \(S\) is a single \(H_r\)-orbit.
\end{proof}

T

Let $p$ be an odd prime, $m\geq 1$ an integer, and $r$ a positive divisor of $p-1$ such that $(p^m,r)\neq (p,p-1)$. The digraph $G(p^m,r)$ is defined as the circulant digraph
\[
G(p^m,r) = \operatorname{Cay}(\mathbb{Z}_{p^m}, H_r),
\]
where $H_r$ is the unique subgroup of the multiplicative group $\mathbb{Z}_{p^m}^\times$ of order $r$.

\begin{cond}\label{p-power-normal-2dt-cond}
Let \(\Gamma=G(p^m,r)=\operatorname{Cay}(\mathbb Z_{p^m},H_r)\), where
\(p\equiv 3\pmod 4\) is an odd prime, \(m\ge 1\), \(r\) is an odd divisor
of \(p-1\) satisfying the following three properties:
\[
H_r\cap(-H_r)=\varnothing,\qquad
H_r\cap(H_r+H_r)\neq\varnothing,\qquad
|H_r+H_r|=2r.
\]
\end{cond}

\begin{prop}\label{p-power-normal-2dt}
Let \(\Gamma=G(p^m,r)=\operatorname{Cay}(\mathbb Z_{p^m},H_r)\), where
\(p\) is an odd prime, \(m\ge 1\), and \(r\ge 2\) divides \(p-1\) with   \((p^m,r)\neq (p,p-1)\).
Assume further that \(H_r\cap(-H_r)=\varnothing\).
Then \(\Gamma\) is \(2\)-distance-transitive if and only if \(\Gamma\) satisfies Condition~\ref{p-power-normal-2dt-cond}.
\end{prop}

\begin{proof}
Consider the Cayley digraph   \(\Gamma=G(p^m,r)=\operatorname{Cay}(\mathbb Z_{p^m},H_r)\). The condition
\(H_r\cap(-H_r)=\varnothing\) implies \(-1\notin H_r\), which readily yields that \(r\) is odd.
By Theorem~\ref{at-primepower-normalcirc}, the digraph \(G(p^m,r)\) is
a normal, arc-transitive circulant digraph of order \(p^m\), with both in-valency and out-valency
equal to \(r\). Its automorphism group admits the structure
\[
A=\operatorname{Aut}(G(p^m,r))\cong \mathbb Z_{p^m}\rtimes H_r,
\]
which acts regularly on the arc set of $\Gamma$.

\medskip
\noindent\textbf{Sufficiency.}
Suppose that \(\Gamma\) satisfies Condition~\ref{p-power-normal-2dt-cond}.
Then \(H_r\cap(H_r+H_r)\neq\varnothing\). As  \(H_r\) forms   a multiplicative
subgroup, this nonempty intersection  implies \(H_r\subseteq H_r+H_r\). Indeed, if
\(x\in H_r\cap(H_r+H_r)\), say \(x=h_1+h_2\) with \(h_1,h_2\in H_r\), then for
any \(g\in H_r\),
\[
g=(gx^{-1})x=(gx^{-1})h_1+(gx^{-1})h_2\in H_r+H_r.
\]
Thus
\[
|(H_r+H_r)\setminus H_r|
=|H_r+H_r|-|H_r|
=2r-r=r.
\]
By Lemma~\ref{2d-1-orbit}, the set \((H_r+H_r)\setminus H_r\) constitutes  a single orbit
under the multiplicative action of \(H_r\) on \(\mathbb Z_{p^m}\). Accordingly,  the
vertex stabilizer \(A_u\cong H_r\) acts transitively on the set of vertices at directed distance \(2\)
from \(u\). Combined with the arc-transitivity of  \(\Gamma\), this guarantees  that \(\Gamma\) is
\(2\)-distance-transitive.

\medskip
\noindent\textbf{Necessity.}
Conversely, assume   that \(\Gamma\) is \(2\)-distance-transitive. As a normal circulant digraph,  $\Gamma$ satisfies  \(p\equiv 3\pmod 4\) by    Proposition \ref{2dtdig-normal-primepower}. Let \(u=0\) be   the
identity element of \(\mathbb Z_{p^m}\). Since \(A_u\cong H_r\) and \(A\) acts
regularly on the arc set, \(A_u\) acts regularly on the out-neighbourhood
\(\Gamma^+(0)=H_r\).

We now claim that \(A_u\)  acts regularly on
\(\Gamma_2^+(0)\), the set of vertices at directed distance \(2\) from \(u\).
It is straightforward to see the set inclusion
\[
\Gamma_2^+(0)\subseteq (H_r+H_r)\setminus H_r.
\]
By Lemma~\ref{2d-1-orbit} and  its proof, the assumption
\(H_r\cap(-H_r)=\varnothing\) ensures  that every element of \(H_r+H_r\) is a
unit in \(\mathbb Z_{p^m}\). Suppose  \(g\in H_r\) and
\(x\in\Gamma_2^+(0)\) satisfy \(gx=x\). Then \((g-1)x=0\), and since \(x\) is a
unit, we obtain \(g=1\). Because \(\Gamma\) is \(2\)-distance-transitive,
\(A_u\) acts transitively on \(\Gamma_2^+(0)\), and so the action is regular.
Consequently,
\[
|\Gamma_2^+(0)|=|H_r|=r.
\]

Now choose \(v\in\Gamma^+(0)=H_r\). We proceed by considering two cases based on
whether \(H_r\cap(H_r+H_r)\) is empty.

First, suppose that \(H_r\cap(H_r+H_r)=\varnothing\). Then for any  \(v\in H_r\), we
have \(H_r\cap(v+H_r)=\varnothing\); otherwise, multiplying by \(v^{-1}\)
would yield \(H_r\cap(1+H_r)\neq\varnothing\), which is equivalent to
\(H_r\cap(H_r+H_r)\neq\varnothing\). Hence
\[
\Gamma^+(0)\cap \Gamma^+(v)=H_r\cap(v+H_r)=\varnothing.
\]
Thus every out-neighbour of \(v\) lies at distance \(2\) from \(u\), and so
\[
\Gamma^+(v)\subseteq \Gamma_2^+(0).
\]
Since \(|\Gamma^+(v)|=r=|\Gamma_2^+(0)|\), we obtain
\[
\Gamma_2^+(0)=\Gamma^+(v).
\]
By Proposition~\ref{empty-equal-1}, this forces
\(\Gamma\cong C_s(n,1)\) for some integers \(n\ge 2\), \(s\ge 3\), with
\(ns=p^m\). However, by Lemma~\ref{crv-not-circdig-1}, \(C_s(n,1)\) is not
normal, contradicting that \(\Gamma\) is normal.

Therefore \(H_r\cap(H_r+H_r)\neq\varnothing\). Then, as in the sufficiency
part, the nonempty intersection implies \(H_r\subseteq H_r+H_r\). Moreover,
every vertex at distance \(2\) from \(u\) is a sum of two elements of \(H_r\)
and does not belong to \(H_r\), hence
\[
\Gamma_2^+(0)=(H_r+H_r)\setminus H_r.
\]
Therefore
\[
|(H_r+H_r)\setminus H_r|=|\Gamma_2^+(0)|=r.
\]
Since \(H_r\subseteq H_r+H_r\), we get
\[
|H_r+H_r|=|H_r|+r=2r.
\]

 Thus Condition~\ref{p-power-normal-2dt-cond} holds.
This completes the proof.
\end{proof}

\begin{theo}\label{2dt-normal-them-1}
Let \(\Gamma=\operatorname{Cay}(T,S)\) be a \(2\)-distance-transitive circulant digraph, where \(T\cong \mathbb{Z}_n\), \(S\subseteq T\setminus\{0\}\), \(-S\cap S=\varnothing\), and \(\langle S\rangle=T\). Assume that \(R(T)\) is normal in \(\operatorname{Aut}(\Gamma)\). Then \(\Gamma\) is isomorphic to one of the following oriented graphs:
\begin{enumerate}[{\rm (1)}]
    \item \(\overrightarrow{C_n}\);
    \item \(G(p^m,r)\) satisfying Condition~\ref{p-power-normal-2dt-cond}.
\end{enumerate}
\end{theo}

\begin{proof}
If \(|S|=1\), then \(\Gamma\cong \overrightarrow{C_n}\). Assume now that \(|S|\ge 2\). Since \(\Gamma\) is a \(2\)-distance-transitive circulant digraph and \(R(T)\trianglelefteq \operatorname{Aut}(\Gamma)\), if \(n\) is even, then Proposition \ref{lem:even-3} implies that  \(\Gamma\) is isomorphic to $\overrightarrow{C_n}$.

If \(n\) is odd, then Proposition \ref{2dtdig-normal-primepower} yields \(n=p^m\) for some odd prime \(p\) with \(p\equiv 3\pmod{4}\) and some \(m\ge 1\).

By Theorem \ref{at-primepower-normalcirc}, \(\Gamma\cong G(p^m,r)\cong \operatorname{Cay}(\mathbb{Z}_{p^m},S)\) for some divisor \(r\) of \(p-1\), where \(S=H_r\) is the unique subgroup of \(\mathbb{Z}_{p^m}^\times\) of order \(r\). Since \(-S\cap S=\varnothing\), \(G(p^m,r)\) is directed, applying Theorem \ref{at-primepower-normalcirc} again shows that \(r\) is odd.

Since \(\Gamma=G(p^m,r)\cong \operatorname{Cay}(T,S)\) is \(2\)-distance-transitive and \(S\cap(-S)=\varnothing\), Proposition \ref{p-power-normal-2dt} implies that \(\Gamma\) satisfies Condition~\ref{p-power-normal-2dt-cond}.
\end{proof}

\bigskip

\bigskip

\bigskip

\section{Proof of Main Theorem}

In this section, we establish the main theorem of this paper.
We use $\K_n$ and $\overline{\K_n}$ to denote the complete graph and the edgeless graph of order $n$, respectively.

The following proposition indicates that
every  connected \(2\)-distance-transitive circulant oriented graph is isomorphic to the lexicographic product of  a smaller   normal \(2\)-distance-transitive circulant oriented graph and  an edgeless graph.

\begin{prop}\label{prop:no-Kn1-factor}
Let \(\Gamma\) be a connected \(2\)-distance-transitive circulant oriented graph. Consider the decomposition
\(
\Gamma\cong
(\Gamma_0\times \K_{n_1}\times\cdots\times \K_{n_r})[\overline{\K}_b]
\)
yielded by  Theorem~\ref{arccirculant-explicit-1}. Then \(r=0\), which implies
\(
\Gamma\cong \Gamma_0[\overline{\K}_b],
\)
where \(\Gamma_0\) is a normal  \(2\)-distance-transitive circulant oriented graph.
\end{prop}

\begin{proof}
By Lemma \ref{prop:Gamma0-2DT}, \(\Gamma_0\) is \(2\)-distance-transitive.

Assume for contradiction that \(r\ge 1\). Put
\[
H=\Gamma_0\times \K_{n_1}\times\cdots\times \K_{n_r},
\]
so that
\(
\Gamma\cong H[\overline{\K}_b].
\)

First, \(H\) has no \(2\)-cycle. Indeed, if \(h\to h'\to h\) were a \(2\)-cycle in \(H\), then for every \(i\in V(\overline{\K}_b)\) we would have
\[
(h,i)\to (h',i)\to (h,i)
\]
in \(\Gamma=H[\overline{\K}_b]\), contradicting the assumption that \(\Gamma\) has no \(2\)-cycle.

Since \(r\ge 1\), the direct product \(H\) contains the factor \(\K_{n_1}\), which is a complete digraph and hence has a \(2\)-cycle. In a direct product of digraphs, a \(2\)-cycle exists if and only if every factor has a \(2\)-cycle. Thus, because \(H\) has no \(2\)-cycle, the factor \(\Gamma_0\) has no \(2\)-cycle.

The digraph \(\Gamma_0\) is connected by Theorem~\ref{arccirculant-explicit-1}. Since   \(\Gamma_0\) has no \(2\)-cycle, it is not a complete digraph. Hence there exist \(u,v\in V(\Gamma_0)\) such that \(u\not\to v\). By connectedness, there is a directed path from \(u\) to \(v\). Choose a shortest such path, say
\[
u=u_0\to u_1\to \cdots \to u_k=v.
\]
Since \(u\not\to v\), we have \(k\ge 2\). Then
\[
u_0\to u_1\to u_2,
\]
and \(u_0\not\to u_2\); otherwise the path could be shortened. Thus \((u_0,u_2)\) is an ordered pair at distance \(2\) in \(\Gamma_0\). Renaming \(u_0,u_1,u_2\) as \(u,w,v\), we have
\[
u\to w\to v
\]
and \(u\not\to v\) in \(\Gamma_0\).

Fix such \(u,w,v\). Since \(n_i\ge 4\), for every \(i=1,\dots,r\) and every \(a,b\in V(\K_{n_i})\), there exists \(c\in V(\K_{n_i})\) such that
\[
a\to c\to b
\]
in \(\K_{n_i}\). Therefore, in \(H\), for any \(a,b\in V(\K_{n_1})\) and any fixed coordinates in the other factors, the vertices
\[
x=(u,a,x_2,\dots,x_r),\qquad
y=(v,b,y_2,\dots,y_r)
\]
satisfy
\[
x\to (w,c,z_2,\dots,z_r)\to y
\]
for suitable \(c,z_2,\dots,z_r\). Moreover, \(x\not\to y\), because \(u\not\to v\) in \(\Gamma_0\). Thus \((x,y)\) is an ordered pair at distance \(2\) in \(H\).

Now take
\[
x_0=(u,0,0,\dots,0),\qquad y_0=(v,0,0,\dots,0),
\]
and
\[
x_1=(u,0,0,\dots,0),\qquad y_1=(v,1,0,\dots,0).
\]
By the previous paragraph, both \((x_0,y_0)\) and \((x_1,y_1)\) are ordered pairs at distance \(2\) in \(H\). In the first pair the \(\K_{n_1}\)-coordinates are equal, while in the second pair they are distinct.

By Theorem~\ref{arccirculant-explicit-1},
\[
\operatorname{Aut}(\Gamma)\cong
S_b\wr
\bigl(\operatorname{Aut}(\Gamma_0)\times S_{n_1}\times\cdots\times S_{n_r}\bigr).
\]
Since \(\Gamma=H[\overline{\K}_b]\), its automorphism group is \(S_b\wr \operatorname{Aut}(H)\). Comparing the base groups gives
\[
\operatorname{Aut}(H)\cong
\operatorname{Aut}(\Gamma_0)\times S_{n_1}\times\cdots\times S_{n_r}.
\]
Suppose that \((x_0,y_0)\) and \((x_1,y_1)\) lie in the same orbit under \(\operatorname{Aut}(H)\). Then there is \(\varphi\in\operatorname{Aut}(H)\) such that
\[
\varphi(x_0)=x_1,\qquad \varphi(y_0)=y_1.
\]
But \(x_0=x_1\), so \(\varphi\) fixes \(x_0\). Write
\[
\varphi=(\varphi_0,\sigma_1,\dots,\sigma_r)
\in
\operatorname{Aut}(\Gamma_0)\times S_{n_1}\times\cdots\times S_{n_r}.
\]
Since \(\varphi(x_0)=x_0\), we have \(\varphi_0(u)=u\) and \(\sigma_1(0)=0\). Hence
\[
\varphi(y_0)=(\varphi_0(v),\sigma_1(0),0,\dots,0)
=(\varphi_0(v),0,0,\dots,0),
\]
which cannot equal \(y_1=(v,1,0,\dots,0)\). This contradiction shows that \((x_0,y_0)\) and \((x_1,y_1)\) are not in the same orbit under \(\operatorname{Aut}(H)\).

On the other hand, \(H\) is \(2\)-distance-transitive. Indeed, if two ordered pairs at distance \(2\) in \(H\) were not in the same orbit under \(\operatorname{Aut}(H)\), then taking the corresponding pairs inside a single fibre of \(\Gamma=H[\overline{\K}_b]\) would give two ordered pairs at distance \(2\) in \(\Gamma\) that are not in the same orbit under \(\operatorname{Aut}(\Gamma)\cong S_b\wr \operatorname{Aut}(H)\). This contradicts the \(2\)-distance-transitivity of \(\Gamma\). Thus \(H\) is \(2\)-distance-transitive.

This contradicts the conclusion of the previous paragraph. Therefore the assumption \(r\ge 1\) is impossible, so \(r=0\). Consequently,
\(
\Gamma\cong \Gamma_0[\overline{\K}_b],
\)
and \(\Gamma_0\) is \(2\)-distance-transitive. This proves the proposition.
\end{proof}

\medskip

We are now ready to prove Theorem~\ref{th-2dt-circ}.

\medskip
\noindent{\bf Proof of Theorem~\ref{th-2dt-circ}.}
Let $T$ be a cyclic group of order $n\ge 3$ and let $S\subseteq T$ satisfy $T=\langle S\rangle$. Let $\Gamma=\operatorname{Cay}(T,S)$ be a 2-distance-transitive  circulant digraph. Then $\Gamma$ is connected.

The right regular representation $R(T)\leq G=\Aut(\Gamma)$ acts  transitively on $V(\Gamma)$, so $\Gamma$ is vertex-transitive. Let $u=0_T$ be the identity element of $T$. Then $\Gamma^+(u)=S$. Since $\Gamma$ is 2-distance-transitive, $G_u$ acts transitively  on $\Gamma^+(u)$,   and hence $\Gamma$ is arc-transitive. Therefore either $S=-S$ or $S\cap -S=\varnothing$.

If  $S=-S$, then by \cite{CJL-2019},   $\Gamma$ is isomorphic to one of the following graphs:
the cycle  $C_n$, the complete $\K_n$,  the complete bipartite $\K_{\frac{n}{2},\frac{n}{2}}$,  the complete multipartite graph $\K_{m[b]}$ with $m\geq 3,b\geq 2$ and $mb=n$,  the graph $\K_{\frac{n}{2},\frac{n}{2}}-\frac{n}{2}\K_2$, where $\frac{n}{2}$ is odd, and Paley graphs of prime order.

We now consider the case $S\cap -S=\varnothing$. If  $|S|=1$, then connectedness forces $\Gamma$ to be the directed cycle $\vec{C}_n$, giving part (II)(1).

Assume henceforth that $|S|\geq 2$. If $R(T)$ is normal in $G$, then by Theorem~\ref{2dt-normal-them-1},  $\Gamma$ is isomorphic to one of the following oriented graphs.
\begin{enumerate}
    \item $\overrightarrow{C_n}$.
    \item   $ G(p^m,r)$  satisfying Condition~\ref{p-power-normal-2dt-cond}.

\end{enumerate}

If $R(T)$ is not normal in $G$, then
by Proposition \ref{prop:no-Kn1-factor},
 \(\Gamma\cong
\Sigma[\overline{\K}_b]\), where $b\geq 2$ and  \(\Sigma\) is a \(2\)-distance-transitive normal circulant oriented graph as above.
In particular, if $\Sigma\cong \overrightarrow{C_r}$ for some $r\geq 3$, then
\(\Gamma\cong
\overrightarrow{C_r}[\overline{\K}_b]\cong C_r(b,1)\).

This completes the proof.
\qed

\end{document}